\documentclass[a4]{amsart}

\usepackage{graphics,amsmath,amssymb,amsthm,mathrsfs}
\usepackage[showonlyrefs]{mathtools}
\usepackage{hyperref}
\usepackage{graphicx,color}
\usepackage{xcolor}

\usepackage{fullpage}

\newtheorem{thm}{Theorem}[section]
\newtheorem{lemma}[thm]{Lemma}
\newtheorem{prop}[thm]{Proposition}
\newtheorem{cor}[thm]{Corollary}
\theoremstyle{definition}
\newtheorem{remark}[thm]{Remark}
\newtheorem*{defn}{Definition}

\makeatletter
\@namedef{subjclassname@2020}{%
  \textup{2020} Mathematics Subject Classification}
\makeatother

\def\R{\mathbb{R}}

\numberwithin{equation}{section}

\title{Least energy positive solutions of critical Schr\"{o}dinger systems with mixed competition and cooperation terms: the higher dimensional case}

\author[H. Tavares]{Hugo Tavares}  \thanks{H.~Tavares is partially supported by the Portuguese government through FCT-Funda\c c\~ao para a Ci\^encia e a Tecnologia, I.P., under the projects UID/MAT/04459/2020, PTDC/MAT-PUR/28686/2017 and PTDC/MAT-PUR/1788/2020.}
\address{Hugo Tavares \newline \indent CAMGSD and Mathematics Department,
\newline \indent Instituto Superior T\'ecnico, Universidade de Lisboa   \newline \indent
Av. Rovisco Pais, 1049-001 Lisboa, Portugal}
\email{hugo.n.tavares@tecnico.ulisboa.pt}
\author[S. You]{Song You}
\author[W.-M. Zou]{Wenming Zou}\thanks{W. M. Zou is supported by NSFC (No. 11771234, No. 12026227).}
\address{Song You and Wenming Zou \newline \indent Department of Mathematical Sciences, \newline \indent Tsinghua University, Beijing 100084, China}
\email{yousong@mail.tsinghua.edu.cn (S. You), zou-wm@mail.tsinghua.edu.cn (W. M. Zou)}

\begin{document}
\date{\today}
\begin{abstract}
Let $\Omega\subset \mathbb{R}^{N}$ be a smooth bounded domain.  In this paper we investigate the existence of least energy positive solutions to the following Schr\"{o}dinger system with $d\geq 2$ equations
\begin{equation*}
-\Delta u_{i}+\lambda_{i}u_{i}=|u_{i}|^{p-2}u_{i}\sum_{j = 1}^{d}\beta_{ij}|u_{j}|^{p} \text{ in } \Omega, \quad u_i=0 \text{ on } \partial \Omega, \qquad i=1,...,d,
\end{equation*}
in the case of a critical exponent  $2p=2^*=\frac{2N}{N-2}$ in high dimensions $N\geq 5$. We treat the focusing case  ($\beta_{ii}>0$ for every $i$)  in the variational setting $\beta_{ij}=\beta_{ji}$ for every $i\neq j$, dealing with a Br\'ezis-Nirenberg type problem: $-\lambda_{1}(\Omega)<\lambda_{i}<0$, where $\lambda_{1}(\Omega)$ is the first eigenvalue of $(-\Delta,H^1_0(\Omega))$. We provide several sufficient conditions on the coefficients $\beta_{ij}$ that ensure the existence of least energy positive solutions; these include the situations of pure cooperation ($\beta_{ij}> 0$ for every $i\neq j$), pure competition ($\beta_{ij}\leq 0$ for every $i\neq j$) and  coexistence of both cooperation and competition coefficients.

In order to provide these results, we firstly establish the existence of nonnegative solutions with $1\leq m\leq d$ nontrivial components by comparing energy levels of the system with those of appropriate sub-systems. Moreover, based on new energy estimates, we obtain the precise asymptotic behaviours of the nonnegative solutions in some situations which include an analysis of a phase separation phenomena.  Some proofs depend heavily on the fact that $1<p<2$, revealing some different phenomena comparing to the special case $N=4$. Our results provide a rather complete picture in the particular situation where the components are divided in two groups.  Besides, based on the results about the phase separation, we prove the existence of least energy sign-changing solution to the Br\'ezis-Nirenberg problem
\[
-\Delta u+\lambda u=\mu |u|^{2^*-2}u,\quad u\in H^1_0(\Omega),
\]
for $\mu>0$, $-\lambda_1(\Omega)<\lambda<0$ for all $N\geq 4$, a result which is new in dimensions $N=4,5$.
\end{abstract}


\maketitle
\section{\bf Introduction}

Consider the following nonlinear elliptic system with $d\geq 2$ equations:
\begin{equation}\label{S-system}
\begin{cases}
-\Delta u_{i}+\lambda_{i}u_{i}=u_{i}^{p-1}\sum_{j = 1}^{d}\beta_{ij}  u^p_{j}   ~\text{ in } \Omega,\\
u_{i}\geq 0 \text{ in } \Omega, \quad u_{i}=0 \text{ on } \partial\Omega,  \quad i=1,...d,
\end{cases}
\end{equation}
when $\Omega$ is a bounded smooth domain of $\R^N$, $N\geq3$, in a critical or subcritical regime $2<2p\leq2^*:=2N/(N-2)$,  and $\lambda_i>-\lambda_1(\Omega)$, where $\lambda_{1}(\Omega)$ is the first eigenvalue of $(-\Delta, H_{0}^{1}(\Omega))$. The solutions of system \eqref{S-system} are related to the system of Gross--Pitaevskii equations
\begin{equation}\label{S-system-2}
\begin{cases}
\imath \partial_t \Phi_i + \Delta \Phi_i + \Phi_i |\Phi_i|^{p-2}\sum_{j=1}^d\beta_{ij}  |\Phi_j|^p=0, \\
\Phi_i=\Phi_i(t,x),\quad i=1,...d,
\end{cases}
\end{equation}
where $\imath$ is the imaginary unit. This system comes from many physical models; for example, system \eqref{S-system-2} can be used to describe Bose-Einstein condensation (see \cite{Timmermans 1998}). In the models, the solutions $\Phi_i(i=1,...d)$ are the corresponding condensate amplitudes, $\beta_{ii}$ represent self-interactions within the same component, while $\beta_{ij}$ $(i\neq j)$ expresses the strength and the type of interaction between different components $i$ and $j$. When $\beta_{ij}>0$ this represents cooperation, while $\beta_{ij}<0$ is of competitive type. In order to acquire standing wave solutions, we denote $\Phi_i(x,t)=e^{\imath \lambda_i t}u_i(x),i=1,...d$. Then $u_i$ satisfies system \eqref{S-system} if and only if $\Phi_i$ satisfies \eqref{S-system-2}.

Denote $\mathbf{H}:=H^1_0(\Omega)\times \cdots \times H^1_0(\Omega)$. Because of the assumption $\beta_{ij}=\beta_{ji}$, by the Sobolev embedding $H^1_0(\Omega)\hookrightarrow L^{2^*}(\Omega)$ and since $2p\leq 2^*$, we know that the solutions of system \eqref{S-system} are critical points of the $C^1$ functional $J:\mathbf{H}\to \mathbb{R}$ given by
\begin{equation}\label{Functional-1}
J(\mathbf{u})=J(u_1,\ldots, u_d):=\frac{1}{2}\sum_{i=1}^{d} \|u_i\|_i^2-\frac{1}{2p}
\int_\Omega \sum_{i,j=1}^{d}\beta_{ij}|u_{i}|^p|u_{j}|^p
\end{equation}
where $\|u_i\|_i^2:=\int_{\Omega} (|\nabla u_{i}|^{2}+\lambda_{i}u_{i}^{2})$. Note, however, that $J$ is not of class $C^2$ for all ranges of $2<2p\leq 2^*$.

\begin{defn}
A vector $\mathbf{u}=(u_{1},\cdots,u_{d})$ is called a \emph{fully nontrivial solution} of \eqref{S-system} if $J'(\mathbf{u})=0$ in $\mathbf{H}^{-1}$ with $u_i\not \equiv 0$ for every $i=1,\ldots,d$. The vector $\mathbf{u}$ is called a \emph{positive solution} of \eqref{S-system} if $\mathbf{u}$ is a solution and $u_i> 0$ for every $i=1,\ldots,d$. A positive solution $\mathbf{u}$ is called a \emph{least energy positive solution} if $J(\mathbf{u})\leq J(\mathbf{v})$ for any positive solution $\mathbf{v}$ of \eqref{S-system}. A solution $\mathbf{u}\not\equiv \mathbf{0}$ is called a \emph{ground-state solution} if $J(\mathbf{u})\leq J(\mathbf{v})$ for any solution $\mathbf{v}$ of \eqref{S-system} with $\mathbf{v} \not \equiv \mathbf{0}$.
\end{defn}

In order to motivate our research, in what follows we quickly review the known results in the literature (describing only  the papers which are more related with our approach and results). The existence of solutions to \eqref{S-system} has been studied intensively in the last two decades  in the \emph{subcritical} case $2p=2^*$, starting with the two equations case $d=2$ in \cite{Ambrosetti 2007,Bartsch-Dancer-Wang 2010,MaiaMontefuscoPellacci, Mandel 2015,Sirakov 2007}.
 In this situation, there is only one interaction constant $\beta=\beta_{12}=\beta_{21}$. For an arbitrary number of equations $d\geq 3$, the problem \eqref{S-system} becomes more complex due to the possible occurrence of simultaneous cooperation and competition, that is, the existence of at least two pairs, $(i_1,j_1)$ and $(i_2,j_2)$, such that $i_1\neq j_1, i_2\neq j_2$, $\beta_{i_1 j_1}>0$ and $\beta_{i_2 j_2}<0$; this study started in \cite{Wei 2005} with $N\geq 3$, $p=2$. Recently, in \cite{BSWang 2016, BLWang 2019, Sato-Wang 2015} the authors investigated the existence and asymptotic behavior of least energy positive solutions for $d=3$ with strong cooperation on a bounded domain.  In \cite{Soave 2015,Soave-Tavares 2016}, by dividing the $d$ components into $m$ groups, the authors proved that system \eqref{S-system} has a least energy positive solution under appropriate assumptions on $\beta_{ij}$. Other related and recent results in the subcritical case can be found in \cite{ClappPistoia, ClappSzulkin, Correia,Correia2, TO 2016, Peng-Wang 2019, Wei-Wu 2020} and references, where we stress that the first two  mentioned papers deal with a general $p$.

For the \emph{critical} case $2p=2^*$, when $d=1$ the system \eqref{S-system} is reduced to the classical Br\'ezis-Nirenberg problem \cite{Brezis 1983}, where the existence of a positive ground state is shown for $-\lambda_1(\Omega)<\lambda_i<0$ when $N\geq 4$. For the $d=2$ equation case, Chen and Zou \cite{Zou 2012} show that there exist $0<\beta_{1}<\beta_{2}$ (depending on $\lambda_{i}$ and $\beta_{ii}$) such that
\begin{equation}\label{Range-1}
\text{system } \eqref{S-system}\text{ has a least energy positive solution if }  \beta_{12}\in (-\infty, \beta_{1}) \cup (\beta_{2}, +\infty)\text{ when } N=4,\ p=2.
\end{equation}
We mention that, when $p=2$ and $\beta\in [\min\{\beta_{11},\beta_{22}\},\max\{\beta_{11},\beta_{22}\}]$, \eqref{S-system} does not have a least energy positive solution.
Still for $d=2$ equations but in the higher dimensional case, Chen and Zou \cite{Zou 2015} proved that
\begin{equation}\label{Range-2}
\text{ system }\eqref{S-system} ~\text{ has a least energy positive solution for any } \beta_{12}\neq 0 ~\text{ when } N\geq 5,\ 2p=2^*.
\end{equation}
In particular, from \eqref{Range-1} and \eqref{Range-2} one deduces that the structure of least energy positive solutions in the critical case changes significantly from $N=4$ to $N\geq 5$. The reason behind this change is the fact that $p\in (1,2)$ whenever $N\geq 5$, while $p=2$ for $N=4$ (in the subcritical case, the importance of this fact has has been implicitly pointed out by Mandel in \cite{Mandel 2015}, see also \cite{TO 2016}). For more results regarding the general critical case with $d=2$ equations, see \cite{Chen-Zou 2014,Pistoia 2018-2,Peng-Wang 2016}.

For three or more equations ($d\geq 3$) in the critical case $2p=2^*$, Clapp and Szulkin \cite{ClappSzulkin} proved that existence of least energy positive solutions for the purely competitive cases in a bounded smooth domain of $\R^N$ with $N\geq 4$ (see also Wu \cite{Wu Y-Z 2017}). Yin and Zou \cite{Yin-Zou}, on the other hand, deal with the purely cooperative case \eqref{S-system} with $N\geq 5$.  Recently, Tavares and You \cite{TS 2019} obtained the existence of least energy positive solutions for the mixed case in a bounded smooth domain of $\R^4$ (see Remark \ref{rem:compare with TY} below for a description of the results from  \cite{TS 2019}).  For other topics related to critical systems (e.g. blowing up solutions as $\lambda_i\to 0^-$ or the case $\lambda_i=0$ in the whole space), see \cite{ChenLin, ClappPistoia, DovettaPistoia, Zou 2018,Yang 2018,  Pistoia 2018-1,Tavares 2017,Tavares 2019}.

We see from \eqref{Range-1} and \eqref{Range-2} that it is meaningful to study the general critical case $2p=2^*$ for \eqref{S-system}, spliting the discussion between the cases $N=4$ and $N\geq 5$. To the best of our knowledge, there are no papers considering the Br\'ezis-Nirenberg type system \eqref{S-system} in the critical case with simultaneous cooperation and competition in the higher dimensional case $N\geq 5$, and this paper gives a first contribution in this direction. We focus on providing conditions on the coefficients that ensure the existence of least energy positive solutions, and we wish to uncover some different phenomena comparing to the special case $N=4$.

Let us now state our structural assumptions. We always assume that
\begin{equation}\label{eq:coefficients}
-\lambda_1(\Omega)<\lambda_1,\ldots, \lambda_d<0,\qquad \Omega \text{ is a bounded smooth domain of } \mathbb{R}^N, \quad 2p=2^*.\tag{H1}
\end{equation}
Moreover,
\begin{equation}\label{eq:beta_ij}
\beta_{ii}>0 \quad \forall i=1,\ldots, d,\qquad \beta_{ij}=\beta_{ji}\quad \forall i, j=1,\ldots, d,\ i\neq j. \tag{H2}
\end{equation}
Next, to state the assumptions related with interactions between different components, we separate the components into $m$ groups, following \cite{Soave 2015,Soave-Tavares 2016, TS 2019}. For any arbitrary $1\leq m \leq d$, we say that a vector $\mathbf{a}=(a_{0},...,a_{m})\in \mathbb{N}^{m+1}$ is an $m$-decomposition of $d$ if
\begin{equation*}
  0=a_{0}<a_{1}<\cdot \cdot \cdot<a_{m-1}<a_{m}=d, \quad \text{ and in this case define } \quad I_{h}:=\left\{i\in \{1,...,d\}:a_{h-1}<i\leq a_{h}\right\}
\end{equation*}
for every $h=1,\ldots, m$. Moreover,
\[
\mathcal{K}_{1}:=\left\{(i,j)\in I_{h}^{2}  \text{ for some } h=1,...,m, \text{ with } i\neq j\right\},\quad \mathcal{K}_{2}:=\left\{(i,j)\in I_{h}\times I_{k} \text{ with } h\neq k\right\}.
\]
In this way, we say that the components $u_i$ and $u_j$ belong to the same group if $(i,j)\in \mathcal{K}_1$, and to a different group if $(i,j)\in \mathcal{K}_2$. In the present paper, we will suppose that the interactions between elements of the same group satisfy the following condition
\begin{equation}\label{Conditions for the same group}
\text{ for any } h=1,\ldots,m, \quad \sum_{(i,j)\in I_h^2}\int_{\Omega}\beta_{ij}|v_i|^p|v_j|^p>0\ \forall \mathbf{v}=(v_{a_{h-1}+1},\ldots, v_{a_h})\in (H^1_0(\Omega))^{|I_h|}\setminus \{\mathbf{0}\} \tag{H3}
\end{equation}
(cf. \cite{Yang 2018}) while  the interaction between elements of different groups is competitive:
\begin{equation}\label{eq:generalcondition_beta}
-\infty<\beta_{ij}\leq 0  \quad\forall (i,j)\in \mathcal{K}_{2}. \tag{H4}
\end{equation}
Observe that condition \eqref{Conditions for the same group} holds if the interaction between elements of the same group is cooperative: $\beta_{ij}\geq 0$ for any $(i,j)\in \mathcal{K}_1$ and $\beta_{ii}>0$ for every $i=1,\ldots, d$.  This condition is actually more general: if the matrix $\mathcal{B}_h:=(\beta_{ij})_{(i,j)\in I_h^2}$ is positive definite (or even copositive definite) for every $h=1,\ldots, m$, then \eqref{Conditions for the same group} is satisfied. Therefore, condition \eqref{Conditions for the same group} also contains the case of simultaneous cooperation and competition within the same group of components.

Next we introduce the main results of this paper.

\subsection{Main results}

Consider the Nehari-type set
\begin{align}\label{Manifold-1}
\mathcal{N}=
\left\{ \mathbf{u}=(u_1,\ldots, u_d)\in \mathbf{H}:\sum_{i\in I_{h}}\|u_{i}\|_{i}\neq 0 \text{ and } \sum_{i\in I_{h}}
\partial_{i}J(\mathbf{u})u_{i}=0 \text{ for every } h=1,...,m
\right\}
\end{align}
and the infimum of $J$ on the set $\mathcal{N}$:
\begin{equation}\label{Minimizer-1}
c:=\inf_{\mathbf{u}\in \mathcal{N}}J\left(\mathbf{u}\right).
\end{equation}

\begin{thm}\label{Theorem-1}
Let $\mathbf{a}$ be an m-decomposition of $d$ for some $1\leq m\leq d$ and assume conditions \eqref{eq:coefficients}--\eqref{eq:generalcondition_beta} hold true for $N\geq 5$.
Then $c$ is attained by a nonnegative $\mathbf{u}\in \mathcal{N}$. Moreover, any minimizer is a  solution of \eqref{S-system} with at least $m$ nontrivial components.
\end{thm}

The proof of Theorem \ref{Theorem-1} is inspired by that of \cite[Theorem 1.3]{Zou 2015} (which deals with $d=2$ equations, $N\geq 5$) and
\cite[Theorem 1.1]{TS 2019}  (which deals with and arbitrary number of equations, but with $N=4$ and $p=2$). The method in \cite{Zou 2015} cannot be used here directly due to the presence of multi-components, and so we need borrow some ideas from \cite{TS 2019} to deal with several groups of components.  There are, however, two key differences with respect with \cite{TS 2019}: here we assume condition \eqref{Conditions for the same group} and not simply $\beta_{ij}>0$ for every $(i,j)\in \mathcal{K}_1$; moreover, $1<p<2$ due to the fact that $N\geq 5$ and $2p=2^*$.

\begin{remark}\label{rem:differencesN}
Related to this last paragraph, we would like to point out that in \cite[Theorem 1.1]{TS 2019} the authors prove a version of  this theorem in the case $N=4$ and $2^*=4$ by considering $\beta_{ij}\leq K$ for every $(i,j)\in \mathcal{K}_{2}$ for some $K>0$, and $\beta_{ij}\geq 0$ for $(i,j)\in \mathcal{K}_1$. The first observation is that in \cite[Theorem 1.1]{TS 2019}, the cooperation assumption in $\mathcal{K}_1$ can be replaced by condition \eqref{Conditions for the same group}. The second observation is that the fact that we only study the case $\beta_{ij}\leq 0$ in $\mathcal{K}_{2}$ in this paper is a direct consequence of the fact that $1<p<2$ in higher dimensions ($N\geq 5$). This is important, for instance, in the proof of Lemma \ref{Bounded} ahead.
\end{remark}

Based on Theorem \ref{Theorem-1}, we prove the existence of least energy positive solutions in several situations.

\begin{thm}\label{thm:leps} Assume that \eqref{eq:coefficients} and \eqref{eq:beta_ij} hold, $N\geq 5$. Then the system admits a least energy positive solution in each one of the following circumstances :
\begin{enumerate}
\item  $\beta_{ij}> 0 \ \forall i,j=1,\ldots, d$, $i\neq j$.
\item  $\beta_{ij}\leq 0 \ \forall i,j=1,\ldots, d$, $i\neq j$.
\item  $\mathbf{a}$ is an $m$-decomposition of $d$ for some $1< m< d$, the condition \eqref{Conditions for the same group} holds true, and
\[
 -\epsilon\leq\beta_{ij}< 0\ \forall (i,j)\in \mathcal{K}_{2},
\]
 for some $\epsilon=\epsilon(\lambda_i,\beta_{ii},(\beta_{ij})_{(i,j)\in \mathcal{K}_1})>0$;
\item  $\mathbf{a}$ is an $m$-decomposition of $d$ for some $1< m< d$, the condition \eqref{Conditions for the same group} holds true,
and for every $M>1$ there exists $b=b(\lambda_i,\beta_{ii},(\beta_{ij})_{(i,j)\in \mathcal{K}_1},M)>0$ such that
\[
\frac{1}{M}\leq \left|\frac{\beta_{i_1j_1}}{\beta_{i_2j_2}}\right|\leq M, ~~\forall (i_1,j_1), (i_2,j_2)\in \mathcal{K}_{2}\quad \text{ and }\quad  \beta_{ij}\leq -b ~~\forall (i,j)\in \mathcal{K}_2.
\]
\end{enumerate}
Moreover, under Case (1), the solution is a ground-state solution.
\end{thm}

We now make a few remarks regarding the four cases present in Theorem \ref{thm:leps}. Case (1) - the purely cooperative case - corresponds to taking a decomposition with $m=1$ group, $\mathbf{a}=(0,d)$, so that the level $c$ reduces to
\begin{equation}\label{ground state}
\inf_{\mathcal{\widetilde N}}J\left(\mathbf{u}\right),\quad \text{ where }  \quad\mathcal{\widetilde N}:= \Big\{\mathbf{u}\in \mathbf{H}:\mathbf{u}\neq \mathbf{0} \text{ and }
J'(\mathbf{u})\mathbf{u}=0 \Big\}.
\end{equation}
Theorem \ref{Theorem-1} shows that this level is achieved. The fact that the solution is in fact fully nontrivial is then a consequence of  having $1<p<2$, which allows to use the reasoning from \cite{TO 2016}. As a consequence, the ground-state level coincides with the least energy positive level.
Observe that Case (1) extends the results about two equations for purely cooperative case present in \cite{Zou 2015} to $d$ equations. The third author together with Yin \cite{Yin-Zou} has obtained the existence of least energy positive solutions for the purely cooperative case, but here we give a more direct proof here depending on Theorem \ref{Theorem-1}.

\smallbreak

Theorem \ref{thm:leps} with Case (2) is a direct consequence of Theorem \ref{Theorem-1} in the particular situation $m=d$ (so that $\mathbf{a}=(0,1,2,\ldots, d)$ necessarily). This result is not new, since it re-obtains  \cite[Theorem 1.3]{ClappSzulkin} (which actually deal with possibly different powers when $N\geq 5$. \smallbreak

Cases (3) and (4), on the other hand,  focus on the mixed coefficients case. Inspired by \cite{Soave 2015}, the strategy of the proof focusses on an asymptotic study of the limits as $\beta_{ij}\to 0^-$ (Theorem \ref{asymtotic behavior-1}) or $\beta_{ij}\to -\infty$ (Theorem \ref{Phase Separation}) for every $(i,j)\in \mathcal{K}_2$. By a contradiction argument, we then prove that  the nonnegative solution obtained in Theorem \ref{Theorem-1} is a positive solution. Moreover, it is a least energy positive solution of \eqref{S-system}. Since system \eqref{S-system} involves critical exponents, we need to establish some new estimates (see Theorem \ref{Energy Estimates}, Theorem \ref{Energy Estimation-1,2,m-1}, Lemma \ref{Energy-3} and Lemma \ref{relationship}) to acquire those precise asymptotic behaviors of the nonnegative solutions.

Case (3) corresponds to a situation where the interaction between elements of the same group contains the cooperative case while  the interaction between elements of different groups is weakly competitive. Case (4), on the other hand, corresponds to a situation where the interaction between elements of the same group contains the pure cooperative case, while the interaction between elements of different groups is strongly competitive.

\begin{remark}
 In the special case $N\leq 3$ and $p=2$, there is no positive solution for some ranges of parameters in the purely cooperative case (see \cite[Proposition 1.1]{Soave 2015}); this result (with the same proof) also holds true in the critical case $N=4$ and $p=2$. However, we deduce from Theorem \ref{thm:leps} (or  \cite{Yin-Zou}) that \eqref{S-system} has a positive solution for any $\beta_{ij}\geq0$ with $N\geq5$. Hence, the general case $N\geq5$ is quite different from the case $N=4$.

\begin{remark}\label{rem:compare with TY}
In \cite[Corollary 1.6, Theorem 1.7, Theorem 1.8]{TS 2019}, the first two authors presented the existence of least energy positive solution under \eqref{eq:coefficients}--\eqref{eq:beta_ij} with  $N=4$  ($p=2$) in each one of the following situation:
\begin{itemize}
\item $-\infty<\beta_{ij}<\Lambda \quad\forall i\neq j$, for some $\Lambda>0$ depending only on $\beta_{ii}, \lambda_{i}$;
\item $\lambda_{i}=\lambda_{h}$ for every $i\in I_{h}, h=1,\ldots, m$; $\beta_{ij}=\beta_{h}>\max\{\beta_{ii}: i\in I_{h}\}$ for every $(i,j)\in I_{h}^{2}$ with $i\neq j, h=1,\ldots, m$; $\beta_{ij}=b<\Lambda$ for every $(i,j)\in \mathcal{K}_{2}$;
\item $\lambda_{i}=\lambda_{h}$ for every $i\in I_{h}, h=1,\ldots, m$; $\beta_{ij}=\beta_{h}>\frac{\alpha}{\alpha-1}\max_{i\in I_h}\{\beta_{ii}\}$ for every $(i,j)\in I_{h}^{2}$ with $i\neq j, h=1,\ldots, m$; $ |\beta_{ij}|\leq \frac{\Lambda}{\alpha d^{2}}$ for every $(i,j)\in \mathcal{K}_{2}$;
\end{itemize}
Observe that in our results for $N\geq 5$, $1<p<2$, we do not require strong cooperation between components belonging to same group ($\beta_{ij}$ for $(i,j)\in \mathcal{K}_1$), nor that $\beta_{ij}=\beta_h$ or $\lambda_{i}=\lambda_{h}$ within the same group $I_h$. With $p=2$, instead, the fact that one asks this is not mere a technical fact, as if the $\lambda_i$'s or the $\beta_{ij}$'s are too far apart, then there are no least energy positive solutions \cite{Tavares 2016}. Check below the proof of Theorem \ref{thm:leps} for the role of $p$. Again, we find interesting to see that some different phenomena appears when comparing with $N=4$.

We would also like to point out that, since $1<p<2$, then $J$ is not of class $C^2$, therefore we cannot use the information coming from the second differential of the energy functional $J$ on $\mathcal{N}$ evaluated in a minimum point $\mathbf{u}$ to deduce existence of least energy positive solution, as done for instance in \cite{Soave 2015, TS 2019}.
\end{remark}

\end{remark}

Based upon Theorem \ref{thm:leps}, it is natural to ask what happen when some interactions between elements of different groups are neither too strong nor too weak. For simplicity and to avoid too technical conditions, we present here the following result for the case of only two groups ($m=2$), for the partition $I_1=\{1,2\}$, $I_2=\{3\}$.

\begin{thm}\label{existence-positive-two groups2}
Assume that $m=2$, $d=3$ and let $0<\sigma_0<\sigma_1$. Then there exists $\widehat{\beta}=\widehat{\beta}(\sigma_i,\beta_{ii},\lambda_i)>0$ such that, if
\[
\beta_{13},\beta_{23} \in [-\sigma_1,-\sigma_0],\qquad \beta_{12}>\widehat{\beta},
\] then system \eqref{S-system} has a least energy positive solution.
\end{thm}

We mention that $\widehat{\beta}$ is defined below in \eqref{Constant9}. Together with Theorem \ref{thm:leps}, our results give a rather complete picture for the case of two groups. 

\smallbreak

\noindent \textbf{Open questions.} Does a least energy positive solution exists under the general conditions: $\beta_{12}> 0$, $\beta_{13},\beta_{23}<0$? And with more components, is is true that a least energy solution exists for $\beta_{ij}>0$ in $\mathcal{K}_1$ and  $\beta_{ij}<0$ in $\mathcal{K}_2$? Is it possible to obtain optimal thresholds?

\smallbreak

As said before, the proofs of cases (3) and (4) in Theorem \ref{thm:leps} are based on an asymptotic study as $\beta_{ij}^n\to 0^-$ and $\beta_{ij}^n\to -\infty$ for $(i,j)\in \mathcal{K}_2$, respectively. Our analysis in the latter case answers some questions left open in the literature, so we decided to state  those asymptotic results in the introduction. This asymptotic limit is related to the study of phase separation, which has attracted a lot of attention in the past 10-15 years. We refer the reader to the paper \cite{STTZ 2016} for a survey. The starting point for a complete characterization  is a uniform $L^\infty$ bound on the solutions. For the subcritical case $2p<2^*$, these are straightforward to obtain from energy bounds (see for instance \cite{NTTG 2010,Soave 2015,  WT 2008, WT 2008-1}). However, for critical case $2p=2^*$, the problem becomes more complicated, see \cite{Zou 2012, Zou 2015, Wu Y-Z 2017} and Remark \ref{2-9} below. In these references, the precise asymptotic behavior of least energy positive solutions was only acquired in some special cases, see Remark \ref{5.1}. In this paper, based on new estimates we can obtain the precise asymptotic behavior of least energy positive solutions for the general case  (which also covers the situation $N=4$).

We introduce some notation used in \cite{Soave 2015} before stating our results.
Define
\begin{equation*}\label{limiting files minimizing}
c^{\infty}:=\inf_{\mathcal{N}^\infty}J^{\infty},\quad \text{ with }\quad J^{\infty}(\mathbf{u}):=\frac{1}{N}\sum_{h=1}^m\|\mathbf{u}_h\|_h^2,
\end{equation*}
and
\begin{equation*}
 \mathcal{N}^{\infty}:=\left\{\mathbf{u}\in \mathbf{H}: \ 0<\sum_{i\in I_h}\|u_i\|_i^2\leq \sum_{(i,j)\in I_h^2}\int_{\Omega}\beta_{ij}|u_{i}|^p|u_{j}|^p\ \forall  h,\ \text{ and }
u_{i}u_{j}\equiv 0\ \forall (i,j)\in \mathcal{K}_2 \right\}.
\end{equation*}

\begin{thm}\label{Phase Separation}
Let $\mathbf{a}$ be an m-decomposition of $d$ for some $1\leq m\leq d$ and assume conditions \eqref{eq:coefficients}--\eqref{Conditions for the same group} hold true for $N\geq 4$. Let $\beta_{ij}^n<0, n\in \mathbb{N}$ satisfy $\beta_{ij}^n\rightarrow -\infty$ as $n\rightarrow\infty$ for every $(i,j)\in\mathcal{K}_2$.
Take $\mathbf{u}^n=(\mathbf{u}_1^n,\ldots, \mathbf{u}_m^n)$ to be a nonnegative solution provided by Theorem \ref{Theorem-1} with $\beta_{ij}=\beta_{ij}^n$ for every $(i,j)\in\mathcal{K}_2$. Therefore, up to a subsequence we have
\begin{equation*}
u_i^n\rightarrow u_i^\infty \text{ strongly in } H^1_0(\Omega)\cap C^{0,\alpha}(\Omega), \forall i,\ \alpha\in (0,1),\qquad \lim_{n\rightarrow\infty}\int_{\Omega}\beta_{ij}^n|u_i^n|^p |u_j^n|^p= 0 \text{ for every } (i,j)\in\mathcal{K}_2,
\end{equation*}
and $c^\infty$ is achieved by $\mathbf{u}^\infty=(\mathbf{u}^\infty_1,\ldots, \mathbf{u}^\infty_m)$, whose components are nonnegative and Lipschitz continuous. Moreover, given $h=1,\ldots, m$, we have $\mathbf{u}^\infty_h\neq 0$,
\begin{equation}\label{limit equation-4}
-\Delta u_i^\infty+\lambda_i u_i^\infty=(u_i^\infty)^{p-1}\sum_{j\in I_h}\beta_{ij}(u_j^\infty)^p \text{ in } \left\{|\mathbf{u}_h^\infty| >0\right\},\qquad \forall i\in I_h,
\end{equation}
and $\overline{\Omega}=\bigcup_{h=1}^m\overline{\left\{|\mathbf{u}_h^\infty| >0\right\}}$.
\end{thm}

The proof of Theorem \ref{Phase Separation} is inspired by that of \cite[Theorem 1.4]{Zou 2012}, although  the method used therein \cite{Zou 2012} cannot be used here directly due to the presence of multi-components, and we need some crucial modifications for our proof (see Theorem \ref{Energy Estimates}, Theorem \ref{Energy Estimation-1,2,m-1} and Lemma \ref{relationship} below).

\begin{remark}\label{2-9}
For any $p$ and under the assumption that solutions of \eqref{S-system} are uniformly bounded in $L^{\infty}(\Omega)$, in \cite{STTZ 2016}, the authors have proved that the above theorem holds true. In \cite{Soave 2015}, the author proves this uniform bound in the subcritical case $N\leq3$, where the compactness of $H^{1}_{0}(\Omega)\hookrightarrow L^{4}(\Omega)$ plays a key role.  In order to prove the result in the critical case, our main contribution is the proof that all limiting components $\mathbf{u}_h^\infty$ are nontrivial, which allows to prove strong convergence in $H^1_0$, thus paving the way (cf. \cite[Lemma 6.1]{Zou 2012}) to the proof of the uniform $L^{\infty}$ boundedness.
\end{remark}

Using Theorem \ref{Phase Separation} in the particular situation $m=d$, we get the following.
\begin{cor}\label{Phase Separation2}
Assume that $N\geq 4, d\geq2, -\lambda_1(\Omega)<\lambda_i<0$. Let $\beta_{ij}^n<0, n\in \mathbb{N}$ and $\beta_{ij}^n\rightarrow -\infty$ when $i\neq j$, and let $(u_1^n,\ldots, u_d^n)$ be a least energy positive solution with $\beta_{ij}=\beta_{ij}^n$. Then, passing to a subsequence we have
\begin{equation*}
u_i^n\rightarrow u_i^\infty \text{ strongly in } H^1_0(\Omega), ~~i=1,\ldots,d, \quad \lim_{n\rightarrow\infty}\int_{\Omega}\beta_{ij}^n|u_i^n|^p |u_j^n|^p= 0, ~\text{ for every } i\neq j,
\end{equation*}
and
\begin{equation*}
u_i^\infty\cdot u_j^\infty\equiv 0 ~\text{ for every } i\neq j,
\end{equation*}
where $u_i^\infty\in C^{0,1}(\overline{\Omega})$ and $\overline{\Omega}=\bigcup_{i=1}^d\overline{\{u_i^\infty>0\}}$.
Moreover, for every $i=1,\ldots,d$, $\{u_i^\infty>0\}$ is a connected domain, and $u_i^\infty$ is a least energy positive solution of
\begin{equation*}
-\Delta u+\lambda_iu=\beta_{ii}|u|^{2^*-2}u, \quad u\in H^1_0(\{u_i^\infty>0\}).
\end{equation*}
\end{cor}

\begin{remark}\label{5.1}
Corollary \ref{Phase Separation2} improves previous results from the literature. It allows to answer a question left open in \cite[Remark 1.4(i)]{Zou 2015}, where the authors dealt with the $d=m=2$ equation case: whether or not the limiting configuration was fully nontrivial. So far this was known to be the case only for  $N\geq 9$ and $\beta_{ij}\equiv \beta$ for every $i\neq j$ (see \cite[Theorem 1.3]{Wu Y-Z 2017}). We have shown the answer is always positive in the general case.
\end{remark}

It is well known that the limiting profile of the Schr\"{o}dinger system is related to the sign-changing solutions of the corresponding elliptic equations (see \cite{Zou 2015,Terracini-Verzini 2009,WT 2008}).
 As an application of Corollary \ref{Phase Separation2}, we can obtain the existence of a least energy sign-changing solution to the Br\'ezis-Nirenberg problem also in some lower-dimensions:
\begin{equation}\label{B-N problem}
-\Delta u+ \lambda u=\mu |u|^{2^*-2}u, \quad u\in H^1_0(\Omega),
\end{equation}
where $\mu>0, -\lambda_1(\Omega)<\lambda<0$ and $N\geq 4$. Here, a sign-changing solution $u$ is called a \emph{least energy sign-changing solution} if $u$ attains the minimal functional energy among all sign-changing solutions of \eqref{B-N problem}.
Then we have
\begin{thm}\label{Sign-changing solutions}
Assume that $N\geq 4$, $d=2$, $\lambda_1=\lambda_2=\lambda\in (-\lambda_1(\Omega),0)$, and $\mu_1=\mu_2=\mu>0$. Let $(u_1^\infty, u_2^\infty)$ be as in Corollary \ref{Phase Separation2} for $d=2$. Then $u_1^\infty - u_2^\infty$ is a least energy sign-changing solution of \eqref{B-N problem} and has two nodal domains.
\end{thm}
\begin{remark}\label{5.2}
The existence of sign-changing solution to the Br\'ezis-Nirenberg problem \eqref{B-N problem} has been studied in \cite{Cerami-Solimini-Struwe 1986,Zou 2015,He-He-Zhang 2020,He-Zhang-Wu-Liang 2018,Schechter-Zou} with $N\geq 4$ for a general domain. For some symmetric domains, see \cite{Atkinson-Brezis-Peletier,Castro-Clapp 2003,Clapp-Weth 2005}. We mention in \cite{Cerami-Solimini-Struwe 1986,Zou 2015} the authors proved the existence of least energy sign-changing solutions for $N\geq 6$. However, there are few results considering the  lower-dimensional situations $(N=4,5)$ in the literature. Recently, Roselli and Willem \cite{Roselli-Willem 2009} proved that \eqref{B-N problem} has a least energy sign-changing solution for $N=5$ with $\lambda\in (-\lambda_1(\Omega),-\overline{\lambda})$, for some $\overline{\lambda}\in(0,\lambda_1(\Omega))$. Here, we improve and extend this result to the case $N\geq 4$.
\end{remark}

\subsection{Structure of the paper}
In Section \ref{sec2}, we study some sub-systems, establish crucial energy estimates and present the proof of Theorem \ref{Theorem-1}. In Section \ref{sec5}, we study asymptotic behaviors of the nonnegative solutions as $\beta_{ij}\to 0$ or $\beta_{ij}^n\to -\infty$ for every $(i,j)\in \mathcal{K}_2$, proving in particular Theorem \ref{Phase Separation}, Corollary \ref{Phase Separation2} and Theorem \ref{Sign-changing solutions}. Section \ref{sec6} is devoted the proofs of Theorem \ref{thm:leps}. In Section \ref{sec8}, we consider the case of two groups and prove Theorem \ref{existence-positive-two groups2}.

\subsection{Further notation}\label{subsec1.4}

\begin{itemize}
  \item The $L^{p}(\Omega)$ norms will be denoted by $|\cdot|_{p}$, $1\leq p\leq \infty$.
  \item Let
  \begin{equation*}\label{eq:def_of_S}
  S:= \inf_{i=1,\ldots,d}\inf_{u\in H^{1}_{0}(\Omega)\setminus \{0\}}\frac{\|u\|^{2}_{i}}{|u|^{2}_{2p}}>0,\quad \text{ so that }\quad   S|u|^{2}_{2p}\leq \|u\|^{2}_{i}\leq \int_{\Omega}|\nabla u|^2 \ \forall u \in H^{1}_{0}(\Omega).
  \end{equation*}

  \item For $1\leq k\leq d$ and $\mathbf{u}=(u_{1},\cdots, u_{k})$, denote $|\nabla \mathbf{u}|^{2}:=\sum_{i=1}^{k}|\nabla u_{i}|^{2}$ and $|\mathbf{u}|^2:=\sum_{i=1}^k|u_{i}|^2$.
  \end{itemize}
  Take $\mathbf{a}=(a_1,\ldots, a_m)$ an $m$-decomposition of $d$, for some integer $m\in [1,d]$.
  \begin{itemize}

    \item Given $\mathbf{u}=(u_1,\ldots, u_d)\in H^{1}_{0}(\Omega; \mathbb{R}^{d})$ we set
    \begin{equation*}
    \mathbf{u}_{h}:=\left(u_{a_{h-1}+1},\ldots,u_{a_{h}}\right)\qquad \text{ for $h=1,\ldots,m$}.
    \end{equation*}
    This way, we have $\mathbf{u}=(\mathbf{u}_1,\ldots, \mathbf{u}_m)$ and $H^1_0(\Omega;\R^d)= \prod_{h=2}^{m} H^1_0(\Omega)^{a_{h}-a_{h-1}}$.  Each space $(H^{1}_{0}(\Omega))^{a_{h}-a_{h-1}}$ is naturally endowed with the following scalar product and norm
    \begin{equation*}
    \langle\mathbf{u}^{1},\mathbf{u}^{2}\rangle_{h}:=\sum_{i\in I_{h}}\langle u^{1}_{i},u^{2}_{i}\rangle_{i} \text{ and }
    \|\mathbf{u}\|^{2}_{h}:=\langle\mathbf{u},\mathbf{u}\rangle_{h}.
    \end{equation*}
    Sometimes we will use the notation $\mathbf{H}= H^{1}_{0}(\Omega; \R^d)$, with norm $\|\mathbf{u}\|^{2}=\sum_{h=1}^{m}\|\mathbf{u}_{h}\|_{h}^{2}$.

    \item Let $\Gamma\subseteq \{1,\ldots,m\}$ and set $d_\Gamma:=|\cup_{k\in\Gamma}I_k|$. For $\mathbf{u}\in H^{1}_{0}(\Omega, \mathbb{R}^{d_\Gamma})$, define the $|\Gamma|\times |\Gamma|$ matrix
    \begin{equation}\label{Matrix-Def-2}
    M_{B}^\Gamma(\mathbf{u})=(  M_{B}^\Gamma(\mathbf{u})_{hk})_{h,k\in \Gamma}:= \left( \sum_{(i,j)\in I_{h}\times I_{k} }\int_{\Omega}\beta_{ij}|u_{i}|^{p}|u_{j}|^{p}\right)_{h, k\in \Gamma}.
    \end{equation}
When $\Gamma=\{1,\ldots,m\}$, we simply denote $M_{B}(\mathbf{u}):=M_{B}^\Gamma(\mathbf{u})$ and $M_{B}(\mathbf{u})_{hk}:= M_{B}^\Gamma(\mathbf{u})_{hk}$
  \end{itemize}

\section{\bf Sub-systems, estimates and proof of Theorem \ref{Theorem-1}}\label{sec2}
In this section we study the ground state level of sub-systems and present some preliminary results, which are used to prove Theorem \ref{Theorem-1} at the end of this section. From now on, we take $\mathbf{a}=(a_1,\ldots, a_m)$, an $m$-decomposition of $d$ for some integer $m\in [1,d]$. In this way, we fix the sets $\mathcal{K}_1$ and $\mathcal{K}_2$. Throughout this section we always assume  \eqref{eq:coefficients}--\eqref{eq:generalcondition_beta} with $N\geq 5$, hence we ommit the assumtions from the upcoming statements.
\subsection{Sub-systems in the whole space}
Fix $h\in \{1,\ldots, m\}$. Inspired by \cite{TS 2019}, we consider the following sub-system
\begin{equation}\label{sub-system}
\begin{cases}
-\Delta v_{i}=\sum_{j \in I_{h}}\beta_{ij}|v_{j}|^{p}|v_{i}|^{p-2}v_{i}  ~\text{ in } \mathbb{R}^{N},\\
v_{i}\in \mathcal{D}^{1,2}(\mathbb{R}^{N})  \quad \forall i\in I_{h}.
\end{cases}
\end{equation}
When $m=d=1$, \eqref{sub-system} is reduced to the following limiting equation
\begin{equation*}\label{eq:single_critical}
-\Delta u=|u|^{2^*-2}u \qquad \text{ in } \R^N,
\end{equation*}
whose positive solutions in $\mathcal{D}^{1,2}(\R^N)$ are given by
\begin{equation}\label{Class-2}
U_{\sigma,y}(x):=\left(N(N-2)\right)^{\frac{N-2}{4}}\left(\frac{\sigma}{\sigma^{2}+|x-y|^{2}}\right)^{\frac{N-2}{2}},\qquad \sigma>0,\ y\in \R^N.
\end{equation}
Set $\mathbb{D}_h:=(\mathcal{D}^{1,2}(\mathbb{R}^{N}))^{a_{h}-a_{h-1}}$ with the norm $\|\mathbf{u}\|_{\mathbb{D}_h}:=\left(\sum_{i\in I_{h}}\int_{\mathbb{R}^{N}}|\nabla u_{i}|^{2}\right)^{\frac{1}{2}}$.

Define the functional
\begin{equation*}
E_{h}(\mathbf{v}):=\int_{\mathbb{R}^{N}}\left(\frac{1}{2}\sum_{i\in I_{h}}|\nabla v_{i}|^{2}-\frac{1}{2p}\sum_{(i,j)\in I_{h}^{2}}
\beta_{ij}|v_{j}|^{p}|v_{i}|^{p}\right),
\end{equation*}
and the energy level
\begin{equation}\label{eq:Ground State-2}
l_{h}:= \inf_{\mathcal{M}_{h}}E_{h},\quad \text{ with } \quad  \mathcal{M}_{h}:=\Big\{\mathbf{v}\in \mathbb{D}_h:\mathbf{v}\neq \mathbf{0} \text{ and }
\langle\nabla E_{h}(\mathbf{v}),\mathbf{v}\rangle=0 \Big\}.
\end{equation}
 It is standard to prove that $l_{h}>0$, and that
\begin{align*}\label{Vector Sobolev Inequality-1}
 l_{h}  = \inf_{\mathbf{v}\in\mathbb{D}_h\backslash \{\mathbf{0}\}}\max_{t>0}E_{h}(t\mathbf{v})
    = \inf_{\mathbf{v}\in\mathbb{D}_h\backslash \{\mathbf{0}\}}\frac{1}{N}\left(\frac{\int_{\mathbb{R}^{N}}\sum_{i\in I_{h}}|\nabla v_{i}|^{2}}{\left(\int_{\mathbb{R}^{N}}\sum_{(i,j)\in I_{h}^{2}}
\beta_{ij}|v_{j}|^{p}|v_{i}|^{p}\right)^{\frac{2}{2^*}}}\right)^{\frac{N}{2}}.
\end{align*}
Therefore we get the following vector Sobolev inequality
\begin{equation*}\label{Vector Sobolev Inequality}
\sum_{(i,j)\in I_{h}^{2}}\int_{\mathbb{R}^{N}}\beta_{ij}|v_{j}|^{p}|v_{i}|^{p}\leq \left(Nl_h\right)^{\frac{2}{2-N}} \left(\sum_{i\in I_{h}}\int_{\mathbb{R}^{N}}|\nabla v_{i}|^{2}\right)^{\frac{2^*}{2}},~\forall \mathbf{v}\in \mathbb{D}_h,
\end{equation*}
which is important to study system \eqref{S-system}.

Finally, consider $f_h: \mathbb{R}^{|I_{h}|}\mapsto \mathbb{R}$ defined by $f_h(x_{a_{h-1}+1}, \cdots, x_{a_h})=\sum_{(i,j)\in I_h}\beta_{ij}|x_{i}|^{p}|x_{j}|^{p}$,
and denote by $\mathcal{X}_h$ the set of solutions to the maximization problem
\begin{equation}\label{Class-1}
f_h(X_{0})=f^h_{max}:=\max_{|X|=1}f_h(X),\quad |X_{0}|=1.
\end{equation}
Then we have the following theorem
\begin{thm}\label{limit system-6-1}
 The level $l_{h}$ is attained by a vector $\mathbf{V}_{h}\in \mathcal{M}_h$. Moreover, any of such minimizers has the form $\mathbf{V}_{h}=X_{0}(f^h_{max})^{-\frac{N-2}{4}}U_{\sigma,y}$, where $X_{0} \in \mathcal{X}_h$, $y\in \R^N$, $\sigma>0$.
\end{thm}

\begin{proof}[\bf Proof of Theorem \ref{limit system-6-1}.]
The proof follows exactly as the one of Theorem 1.9 in \cite{TS 2019}, and so we omit it.
\end{proof}

\subsection{ \bf Preliminary results for subsystems in bounded domains: lower bounds and Palais-Smale conditions}

In this subsection we present some preliminary results, including the construction of a Palais-Smale sequence at level $c$.

Firstly, we consider an analogue of the level $c$ where we just consider some components of the $m$ groups in which we divided $\{1,\ldots,d\}$. Given $\emptyset \neq \Gamma\subseteq \{1,\ldots,m\}$, we define
\begin{equation}\label{Minimizer-4}
c_{\Gamma}=\inf_{\mathbf{u}\in\mathcal{N}_{\Gamma}}J_{\Gamma}(\mathbf{u}),\quad \text{ where } \quad J_\Gamma(\mathbf{u}):=\frac{1}{2}\sum_{k\in \Gamma}\|\mathbf{u}_{k}\|_{k}^{2}-\frac{1}{2p}\sum_{k,l\in \Gamma}\sum_{(i,j)\in I_k\times I_l}\int_{\Omega}\beta_{ij}|u_{i}|^{p}|u_{j}|^{p}
\end{equation}
and
\begin{align*}\label{Minfold-4}
\mathcal{N}_{\Gamma}:=
\Bigg\{(\mathbf{u}_h)_{h\in \Gamma}:    \mathbf{u}_{h} \neq \mathbf{0},\ \sum_{i\in I_{h}}
\partial_{i}J_\Gamma(\mathbf{u})u_{i}=0, \text{ for every } h\in \Gamma
\Bigg\}.
\end{align*}
Observe that $c=c_{\{1,\ldots,m\}}$.

\begin{lemma}\label{Bounded}
Let $\emptyset \neq \Gamma\subseteq \{1,\ldots, m\}$. Then there exists $C_{1}>0$ such that for any $\mathbf{u}\in \mathcal{N}_\Gamma$ we have
\begin{equation*}\label{Bounded-2}
\sum_{i\in I_{h}}|u_{i}|_{2p}^{2}\geq C_{1} \text{ for every } h\in \Gamma,
\end{equation*}
where $C_{1}$ depends on $(\beta_{ij})_{(i,j)
\in  \mathcal{K}_{1}}$, $(\beta_{ii})_i$ and $(\lambda_{i})_i$, but does not depend on $\Gamma$ nor on $(\beta_{ij})_{(i,j)\in \mathcal{K}_2}$.
\end{lemma}
\begin{proof}[\bf{Proof}]
Note that $\beta_{ij}\leq 0$ for every $(i,j)\in \mathcal{K}_{2}$. Then for any $\mathbf{u}\in \mathcal{N}_\Gamma$ and for any $h\in \Gamma$ we have
\begin{align*}
 S\sum_{i\in I_{h}}|u_{i}|_{2p}^{2}& \leq\sum_{i\in I_{h}}\|u_{i}\|_{i}^{2} \leq \sum_{(i,j)\in I_h^2} \int_{\Omega}\beta_{ij}|u_{i}|^p|u_{j}|^p \leq \frac{1}{2}\sum_{(i,j)\in I_h^2} \int_{\Omega}\beta_{ij}^{+}(|u_{i}|^{2p}+|u_{j}|^{2p}) \\
   &\leq \max_{(i,j)\in I_h^2}\{\beta_{ij}^{+}\}\sum_{i\in I_h}|u_{i}|^{2p}_{2p}\leq d\max_{(i,j)\in I_h^2}\{\beta_{ij}^{+}\}\left(\sum_{i\in I_h}|u_{i}|^2_{2p}\right)^p,
\end{align*}
which implies that
\begin{equation*}
\sum_{i\in I_{h}}|u_{i}|_{2p}^{2}\geq  \left(\frac{S}{ \displaystyle d\max_{(i,j)\in I_h^2}\{\beta_{ij}^{+}\}}\right)^{\frac{1}{p-1}}\geq \min_{1\leq h \leq m}\left\{\left(\frac{S}{ \displaystyle d\max_{(i,j)\in I_h^2}\{\beta_{ij}^{+}\}}\right)^{\frac{1}{p-1}}\right\}=:C_1. \qedhere
\end{equation*}
\end{proof}

Inspired by Proposition 1.2 in \cite{Soave 2015}, we have the following
\begin{lemma}\label{Lagrange}
Let $\emptyset \neq \Gamma\subseteq \{1,\ldots, m\}$ and assume that $c_\Gamma$ is achieved by $\mathbf{u}\in \mathcal{N}_\Gamma$. Then $\mathcal{N}_\Gamma$ is a regular manifold, and $\mathbf{u}$ is a critical point of $J_\Gamma$ with at least $|\Gamma|$ nontrivial components.
\end{lemma}
\begin{proof}[\bf{Proof}]
Without loss of generality, we will present the proof for the case of $\Gamma=\{1,\ldots, m\}$. Define
\begin{equation}\label{Func-1}
\Psi_{k}(\mathbf{u}):= \|\mathbf{u}_{k}\|_{k}^{2}-\sum_{h=1}^m\sum_{(i,j)\in I_{k}\times I_{h}}\int_{\Omega}\beta_{ij}|u_{i}|^{p}|u_{j}|^{p}.
\end{equation}
By a direct computation, for every $\overrightarrow{\psi}=(\overrightarrow{\psi}_1,\cdots, \overrightarrow{\psi}_h,\cdots, \overrightarrow{\psi}_m)\in \mathbf{H}$ we obtain that
\begin{equation}\label{Matrix-def-7}
\langle \Psi'_{k}(\mathbf{u}), \overrightarrow{\psi}\rangle =2\langle\mathbf{u}_k, \overrightarrow{\psi}_k \rangle
-p\sum_{h=1}^m\sum_{(i,j)\in I_k\times I_h}\int_{\Omega}\beta_{ij}\left(|u_i|^p|u_j|^{p-2}u_j\psi_j+|u_j|^p|u_i|^{p-2}u_i\psi_i\right).
\end{equation}
We claim that the set $\mathcal{N}$ defines a smooth manifold of codimension $m$ in a neighbourhood of $\mathbf{u}$ in $\mathbf{H}$. To verify this, we shall prove that
\begin{equation*}
\widehat{\Psi}_{\mathbf{u}}: \mathbf{H}\rightarrow \mathbb{R}^m,
\end{equation*}
is a surjective operator, where
\begin{equation*}
\widehat{\Psi}_{\mathbf{u}}(\mathbf{v}):= ( \Psi'_{1}(\mathbf{u})\mathbf{v},\cdots, \Psi'_{m}(\mathbf{u})\mathbf{v}),~ \mathbf{v}\in \mathbf{H}.
\end{equation*}
 Since $\mathbf{u}\in \mathcal{N}$, then we see that
\begin{equation}\label{Matrix-def-9}
M_{B}(\mathbf{u})_{kk}+\sum_{h\neq k}M_{B}(\mathbf{u})_{kh}=\|\mathbf{u}_{k}\|_{k}^{2},
\end{equation}
where $M_{B}(\mathbf{u})_{kh}$ is defined in \eqref{Matrix-Def-2}.
For any $(t_1,\ldots, t_m)\in \mathbb{R}^m$, by \eqref{Matrix-def-7} and \eqref{Matrix-def-9} we have
\begin{align*}
\langle \Psi'_{k}(\mathbf{u}), (-t_1\mathbf{u}_1,\ldots, -t_m\mathbf{u}_m) \rangle &=\left(-2\|\mathbf{u}_k\|_k^2+2pM_{B}(\mathbf{u})_{kk}+p\sum_{h\neq k}M_{B}(\mathbf{u})_{kh}\right)t_k+p\sum_{h\neq k}M_{B}(\mathbf{u})_{kh}t_h \\
&= \left((2p-2)M_{B}(\mathbf{u})_{kk}+(p-2)\sum_{h\neq k}M_{B}(\mathbf{u})_{kh}\right)t_k+p\sum_{h\neq k}M_{B}(\mathbf{u})_{kh}t_h.
\end{align*}
Define
\begin{equation*}
b_{kk}:=\left[(2p-2)M_{B}(\mathbf{u})_{kk}+(p-2)\sum_{h\neq k}M_{B}(\mathbf{u})_{kh}\right], ~~b_{kh}:= M_{B}(\mathbf{u})_{kh},
\end{equation*}
and
\begin{equation}\label{Matrix-def-2}
\widetilde{\mathbf{B}}:= (b_{kh})_{1\leq k,h\leq m}.
\end{equation}
Hence,
\begin{equation}\label{Matrix-def-10}
(\langle \Psi'_{1}(\mathbf{u}), (-t_1\mathbf{u}_1,\ldots, -t_m\mathbf{u}_m) \rangle,\ldots,\langle \Psi'_{m}(\mathbf{u}), (-t_1\mathbf{u}_1,\ldots, -t_m\mathbf{u}_m) \rangle)^T=\widetilde{\mathbf{B}}\mathbf{t},
\end{equation}
where $\mathbf{t}=(t_1,\ldots,t_m)^T$.
Note that $\beta_{ij}\leq 0, (i,j)\in \mathcal{K}_2$, so that $|M_B(\mathbf{u})_{kh}|=-M_B(\mathbf{u})_{kh}$ for $k\neq h$. Then we deduce from Lemma \ref{Bounded} that
\begin{align*}
b_{kk}-\sum_{h\neq k}|b_{kh}| &= \Bigg[(2p-2)M_{B}(\mathbf{u})_{kk}+(p-2)\sum_{h\neq k}M_{B}(\mathbf{u})_{kh}\Bigg] -p\sum_{h\neq k}\left|M_{B}(\mathbf{u})_{kh}\right|  \\
 & =(2p-2)\Bigg[M_{B}(\mathbf{u})_{kk}+\sum_{h\neq k}M_{B}(\mathbf{u})_{kh}\Bigg]\\
 & =(2p-2)\|\mathbf{u}_k\|_k^2\geq (2p-2)S\sum_{i\in I_k}|u_i|_{2^*}^2\geq (2p-2)SC_{1}.
\end{align*}
Therefore, for every $k=1,\ldots,m$ we see that
\begin{align}\label{Matrix-Positive-1}
b_{kk}-\sum_{h\neq k}|b_{hk}|\geq (2p-2)SC_{1},
\end{align}
where $C_{1}$ is defined in Lemma \ref{Bounded}. It follows from \eqref{Matrix-Positive-1} and $p>1$ that the matrix $\widetilde{\mathbf{B}}$ is strictly diagonally dominant and so $\widetilde{\mathbf{B}}$ is positive definite. Hence, the matrix $\widetilde{\mathbf{B}}$ is non-singular. Thus, for any $\mathbf{h}=(h_1,\ldots,h_m)\in \mathbb{R}^m$, by \eqref{Matrix-def-10} there exists $\overline{\mathbf{t}}=(\overline{t},\ldots,\overline{t}_m)$ such that
\begin{equation*}\label{Matrix-def-11}
(\langle \Psi'_{1}(\mathbf{u}), (-\overline{t}_1\mathbf{u}_1,\ldots, -\overline{t}_m\mathbf{u}_m) \rangle,\ldots,\langle \Psi'_{m}(\mathbf{u}), (-\overline{t}_1\mathbf{u}_1,\ldots, -\overline{t}_m\mathbf{u}_m) \rangle)=\mathbf{h},
\end{equation*}
where $\overline{\mathbf{t}}^T:= \widetilde{\mathbf{B}}^{-1}\mathbf{h}$.
Therefore, the claim is true.

Since the set $\mathcal{N}$ defines a smooth manifold of codimension $m$ in a neighbourhood of $\mathbf{u}$ in $\mathbf{H}$, then by the Lagrange multipliers rule there exist $\mu_1,\ldots, \mu_m\in \mathbb{R}$ such that
\begin{equation}\label{equation-4}
 J'(\mathbf{u})-\sum_{k=1}^m\mu_k \Psi_{k}'(\mathbf{u})=0.
\end{equation}
Note that $J'(\mathbf{u})\mathbf{u}_k=\Psi_{k}(\mathbf{u})=0$, then test \eqref{equation-4} with $(\mathbf{0},\cdots, \mathbf{u}_k, \cdots, \mathbf{0})$, for every $k=1,\ldots,m$, by \eqref{Matrix-def-7} we get that
\begin{equation}\label{Matrix-def-4}
\left((2p-2)M_{B}(\mathbf{u})_{kk}+(p-2)\sum_{h\neq k}M_{B}(\mathbf{u})_{kh}\right)\mu_k+p\sum_{h\neq k}M_{B}(\mathbf{u})_{kh}\mu_h=0.
\end{equation}
Recall that the matrix $\widetilde{\mathbf{B}}$ is defined in \eqref{Matrix-def-2}. We deduce from \eqref{Matrix-def-4} that $\widetilde{\mathbf{B}} \boldsymbol{\mu}=0$, where $\boldsymbol{\mu}=(\mu_1,\cdots,\mu_m)^T$. From the above arguments, we know that the matrix $\widetilde{\mathbf{B}}$ is non-singular, then $\mathbf{A}=\mathbf{0}$ and $\mu_1=\cdots=\mu_m=0$. Combining this with \eqref{equation-4} we have $ J'(\mathbf{u})=0$.
\end{proof}

Finally, we construct a Palais-Smale sequence at level $c_\Gamma$.
\begin{prop}[Existence of a Palais-Smale sequence]\label{PS Sequence}
Let $\emptyset \neq \Gamma\subseteq\{1,\ldots, m\}$. Then there exists a sequence $\{\mathbf{u}_{n}\}\subset\mathcal{N}_\Gamma$ satisfying
\begin{equation*}\label{Inequality16-1}
\lim_{n\rightarrow\infty}J_\Gamma(\mathbf{u}_{n})=c_\Gamma,\quad \lim_{n\rightarrow\infty}J_\Gamma'(\mathbf{u}_{n})=0.
\end{equation*}
\end{prop}
\begin{proof}[\bf{Proof}]
Without loss of generality, we will present the proof for the case of $\Gamma=\{1,\ldots, m\}$.
Notice that $J$ is coercive and bounded from below on $\mathcal{N}$. Then by the Ekeland variational principle (which we can use, since $\mathcal{N}$ is a closed manifold by Lemmas \ref{Bounded} and \ref{Lagrange}), there exists a minimizing sequence $\{\mathbf{u}_{n}\}\subset\mathcal{N}$ satisfying
\begin{equation}\label{Inequality14}
J(\mathbf{u}_{n})\rightarrow c,\quad
J'(\mathbf{u}_{n})-\sum_{k=1}^m\lambda_{k,n}\Psi'_{k}(\mathbf{u}_{n})=o(1), \text{ as } n\rightarrow \infty,
\end{equation}
where the function $\Psi_{k}$ is defined in \eqref{Func-1}.
Obviously, $\{u^{n}_i\}$ is uniformly bounded in $H^1_0(\Omega)$. We suppose that, up to a subsequence,
\begin{equation*}
u_{i}^{n}\rightharpoonup u_{i} \text{ weakly in } H^{1}_{0}(\Omega), \quad
u_{i}^{n}\rightharpoonup u_{i} \text{ weakly in } L^{2^*}(\Omega).
\end{equation*}
Consider $\overline{\mathbf{u}}^k_n$ defined by
\[
\overline{u}^k_{i,n}:=
\begin{cases}
u_{i}^n ~~\text{ if } i\in I_k,\\
0 ~~\text{ if } i\not\in I_k.
\end{cases}
\]
Testing the second equation in \eqref{Inequality14}, for every $k=1,\ldots,m$ we have
\begin{align}\label{Matrix-4}
\left(2\|\mathbf{u}_k^n\|_k^2-2pM_{B}(\mathbf{u}^n)_{kk}-p\sum_{h\neq k}M_{B}(\mathbf{u}^n)_{kh}\right)\lambda_{k,n} -
p\sum_{h\neq k}M_{B}(\mathbf{u}^n)_{kh}\lambda_{h,n}=o(1).
\end{align}
Since $\Psi_{k}(\overline{\mathbf{u}}^k_n)=0$, then we know that
\[
\|\mathbf{u}_k^n\|_k^2=M_{B}(\mathbf{u}^n)_{kk}+\sum_{h\neq k}M_{B}(\mathbf{u}^n)_{kh}.
\]
Combining this with \eqref{Matrix-4}, for every $k=1,\ldots,m$ we get that
\begin{align}\label{Matrix-5}
\left[(2p-2)M_{B}(\mathbf{u}^n)_{kk}+(p-2)\sum_{h\neq k}M_{B}(\mathbf{u}^n)_{kh}\right]\lambda_{k,n}+
p\sum_{h\neq k}M_{B}(\mathbf{u}^n)_{kh}\lambda_{h,n}=o(1).
\end{align}
Define
\begin{equation*}
\widehat{b}_{kk}:= \lim_{n\rightarrow\infty}\left[(2p-2)M_{B}(\mathbf{u}^n)_{kk}+(p-2)\sum_{h\neq k}M_{B}(\mathbf{u}^n)_{kh}\right], ~~\widehat{b}_{kh}:= \lim_{n\rightarrow\infty}M_{B}(\mathbf{u}^n)_{kh},
\end{equation*}
and
\[
\widehat{\mathbf{B}}:= (\widehat{b}_{kh})_{1\leq k,h\leq m}.
\]
Note that $\beta_{ij}\leq 0, (i,j)\in \mathcal{K}_2$ and $\mathbf{u}_n\in \mathcal{N}$. Then similarly to \eqref{Matrix-Positive-1}, for every $k=1,\ldots,m$ we get that
\begin{align*}\label{Matrix-2-7-1}
\widehat{b}_{kk}-\sum_{h\neq k}|\widehat{b}_{hk}|\geq (2p-2)SC_{1},
\end{align*}
where $C_{1}$ is defined in Lemma \ref{Bounded}.
Therefore, $\widehat{\mathbf{B}}$ is positive definite. By \eqref{Matrix-5} we see that
\begin{equation*}
o(1)= (\widehat{\mathbf{B}}+o(1))\widetilde{\mathbf{a}}_n,
\end{equation*}
where $\widetilde{\mathbf{a}}_n:= (\lambda_{k,n})_{m\times 1}$.
Combining this with the argument in \cite[Proposition 3.13]{TS 2019}, we can get that $\mathbf{u}_{n}$ is a standard Palais-Smale sequence.
\end{proof}

\subsection{Energy estimates and proof of Theorem \ref{Theorem-1}}\label{sec3.2}
In this section we show crucial energy estimates for all levels $c_\Gamma$. Then we give the proof of Theorem \ref{Theorem-1} at the end of this subsection.

The following theorem plays an important role in showing that limits of minimizing sequences are nonzero. The presence of a $\delta$ plays a crucial role in proving Theorem \ref{Phase Separation}  and Theorem \ref{thm:leps}-(4).

\begin{thm}\label{Energy Estimates}
There exists $\delta>0$, depending only on $\beta_{ij}$ for $(i,j)\in \mathcal{K}_{1}$ and $\beta_{ii}, \lambda_{i} ~i=1,\ldots,d$, such that
\begin{equation*}
c_\Gamma<\sum_{h\in \Gamma}l_{h}-\delta \leq \sum_{h=1}^ml_h \qquad \forall \emptyset\neq \Gamma\subseteq \{1,\ldots, m\},
\end{equation*}
where $c_\Gamma$ is defined in \eqref{Minimizer-4} and $l_{h}$ is defined in \eqref{eq:Ground State-2}. In particular,
\[
c<\sum_{h=1}^m l_h -\delta.
\]
\end{thm}
\begin{proof}[\bf{Proof}]
Without loss of generality, we show the proof for $\Gamma=\{1,\ldots, m\}$. The following proof is similar to \cite[Theorem 3.6]{TS 2019}, thus we just sketch it here for completeness. Observe that there exist $y_{1}, ..., y_{m}\in \Omega$ and $\rho>0$ such that
$B_{2\rho}(y_{h})\subseteq \Omega$ and $B_{2\rho}(y_{h})\cap B_{2\rho}(y_{k})=\emptyset$ for $h\neq k$. Let $\xi_{h}\in C^{1}_{0}(B_{2\rho}(y_{h}))$ be such that $0\leq\xi_{h}\leq 1$ and $\xi_{h}\equiv 1$ for $|x-y_{h}|\leq\rho$, $h=1,...,m$.
We deduce from Theorem \ref{limit system-6-1} that
\begin{equation*}
\mathbf{V}_{h}=X_0(f_{max}^h)^{-\frac{N-2}{4}}U_{1,0} ~\text{is a minimizer for $l_h$},
\end{equation*}
where $X_{0}, f_{max}^h$ are defined in \eqref{Class-1} and $U_{1,0}$ is defined in \eqref{Class-2}.
Denote
\begin{equation*}
\mathbf{V}_{h}^{\varepsilon}:=\varepsilon^{-\frac{N-2}{2}}X_0(f_{max}^h)^{-\frac{N-2}{4}}U_{1,0}\left(\frac{x-y_{h}}{\varepsilon}\right),\quad
\widehat{\mathbf{V}}_{h}^{\varepsilon}:=\xi_{h}\mathbf{V}_{h}^{\varepsilon}.
\end{equation*}
Then by the proofs of \cite[Lemma 1.46]{Willem 1996} and \cite[Theorem 3.6 - eq. (3.8)]{TS 2019} we get the following inequalities (recall that here $N\geq 5$)
\begin{equation}\label{2.14}
\int_{\Omega}|\widehat{V}_{i}^{\varepsilon}|^{2}\geq C_{i}\varepsilon^{2}+O(\varepsilon^{N-2}), \text{ for every } i\in I_{h},
\end{equation}

\begin{equation*}\label{2.12}
\int_{\Omega}|\nabla \widehat{V}_{i}^{\varepsilon}|^{2}= \int_{\mathbb{R}^{N}}|\nabla V_{i}|^{2}+O(\varepsilon^{N-2}), \text{ for every } i\in I_{h},
\end{equation*}
\begin{equation}\label{2.13}
\int_{\Omega}|\widehat{V}_{i}^{\varepsilon}|^{p}|\widehat{V}_{j}^{\varepsilon}|^{p}= \int_{\mathbb{R}^{N}}| V_{i}|^{p}| V_{j}|^{p}+O(\varepsilon^{N}), \text{ for every } (i,j)\in I_{h}^{2},
\end{equation}
where
\begin{equation}\label{Constant-2-9}
C_{i}=\left(N(N-2)\right)^{\frac{N-2}{2}}(X_{0})_{i}^{2}(f_{max}^h)^{-\frac{N-2}{2}}.
\end{equation}

Recall that $\lambda_{1},\cdots,\lambda_{d}<0$ and $\beta_{ii}>0$. Note that $\int_{\Omega}|\widehat{V}_{i}^{\varepsilon}|^{p}|\widehat{V}_{j}^{\varepsilon}|^{p}dx=0$ for every $(i,j)\in \mathcal{K}_{2}$. Then we deduce from \eqref{2.14}--\eqref{2.13} and $\lambda_i<0$ that, given $t_1,\ldots, t_m>0$,
\begin{align}\label{2.15}
J \left(t_{1}\widehat{\mathbf{V}}_{1}^{\varepsilon},..., t_{m}\widehat{\mathbf{V}}_{m}^{\varepsilon}\right)\leq
 \frac{1}{2}\sum_{h=1}^{m}t_{h}^{2}\left(Nl_{h}-C^h\varepsilon^{2}+ O(\varepsilon^{N-2})\right)-\frac{1}{2^*}\sum_{h=1}^{m}t_{h}^{2^*}\left(Nl_{h}+O(\varepsilon^{N})\right),
\end{align}
where $C^h:= \sum_{i\in I_h}C_i|\lambda_{i}|=\sum_{i\in I_h}\left(N(N-2)\right)^{\frac{N-2}{2}}(X_{0})_{i}^{2}(f_{max}^h)^{-\frac{N-2}{2}}|\lambda_{i}|>0$ and $C_i$ is defined in \eqref{Constant-2-9}.
Denote
\begin{equation*}
A_{h}^{\varepsilon}:=Nl_{h}-C^h\varepsilon^{2}+ O(\varepsilon^{N-2}),\quad
B_{h}^{\varepsilon}:=Nl_{h}+O(\varepsilon^{N}).
\end{equation*}
From the fact that $N\geq 5$ it  is easy to see that
\begin{equation*}
0<A_{h}^{\varepsilon}<B_{h}^{\varepsilon}, \text{ for } \varepsilon \text{ small enough. }
\end{equation*}
Fix such an $\varepsilon$. We have
\begin{align*}\label{2.15-3}
\max_{t_{1},\ldots,t_{m}>0}J \left(t_{1}\widehat{\mathbf{V}}_{1}^{\varepsilon},..., t_{m}\widehat{\mathbf{V}}_{m}^{\varepsilon}\right)
&\leq\max_{t_{1},\ldots,t_{m}>0}\sum_{h=1}^{m}\left( \frac{1}{2}t_{h}^2A_{h}^{\varepsilon}- \frac{1}{2^*}t_{h}^{2^*}B_{h}^{\varepsilon}\right)= \frac{1}{N}\sum_{h=1}^{m}\left(\frac{A_{h}^{\varepsilon}}{B_{h}^{\varepsilon}}\right)^{\frac{2}{2^*-2}}A_{h}^{\varepsilon} \nonumber\\
&< \frac{1}{N}\sum_{h=1}^{m}A_{h}^{\varepsilon}=\sum_{h=1}^{m}\left(l_{h}-\frac{1}{N}C^h\varepsilon^{2}+ O(\varepsilon^{N-2})\right)\nonumber\\
&< \sum_{h=1}^{m}l_{h}-\delta ~\text{ for } \varepsilon \text{ small enough, }
\end{align*}
where
\begin{equation*}\label{4-p-9}
\delta:=\frac{1}{4N}\min_{1\leq h\leq m}\{C^h\}\varepsilon^{2},
\end{equation*}
$\delta$ being a positive constant, only dependent on $\beta_{ij}\geq 0$ for $(i,j)\in \mathcal{K}_{1}$ and $\beta_{ii}, \lambda_{i} ~i=1,\ldots, d$. Next, following the proof of \cite[Theorem 3.6]{TS 2019} we can get that
$c<\sum_{h=1}^{m}l_{h}-\delta$.
\end{proof}

Inspired by \cite[Theorem 3.7]{TS 2019}, we have the following result, which in the next subsection this will play an important role in proving that $c$ is achieved by a solution with $m$ nontrivial components. The presence of a $\delta$ plays a crucial role in proving Theorem \ref{Phase Separation} and Theorem \ref{thm:leps}-(4).

\begin{thm}\label{Energy Estimation-1,2,m-1}
Assume \eqref{eq:generalcondition_beta}. Given $\Gamma\subseteq\{1,\ldots, m\}$, assume that
\[
c_G \text{ is achieved by a nonnegative } \mathbf{u}_G \text{ for every } G\subsetneq \Gamma.
\]
Then
\begin{equation*}
c_\Gamma<\min\left\{  c_G + \sum_{h\in \Gamma\setminus G } l_h-\delta:\ G\subsetneq \Gamma \right\},
\end{equation*}
where $\delta$ is defined in Theorem \ref{Energy Estimates}, depending only on $\beta_{ij}$ for $(i,j)\in \mathcal{K}_{1}$ and $\beta_{ii}, \lambda_{i} ~i=1,\ldots, d$. (and not depending on $\Gamma$ nor on $\beta_{ij}$ with $(i,j)\in \mathcal{K}_2$).
\end{thm}

\begin{remark}\label{Energy estimate-6}
Here, the presence of a $\delta$ in Theorem \ref{Energy Estimates} and Theorem \ref{Energy Estimation-1,2,m-1} will play a critical role in showing the proof of Theorem \ref{Phase Separation} and Theorem \ref{thm:leps}-(4).
\end{remark}

The following proof is inspired by that of \cite[Theorem 3.7]{TS 2019}, and we only sketch it and identify the differences. Without loss of generality, we fix $1\leq q<m$ and only prove that
\begin{equation*}\label{eq:c<c_p}
c<c_{1,\ldots,q}+\sum_{h=q+1}^{m}l_{h}-\delta
\end{equation*}
where we use the notation $c_{1,\ldots, q}$ instead of $c_{\{1,\ldots,q\}}$ for simplicity.

Assume \eqref{eq:generalcondition_beta}, which in particular implies that $\beta_{ij}\leq 0$ for every $(i,j)\in \mathcal{K}_2$. Suppose that $c_{1,\ldots,q}$ is attained by a nonnegative $\mathbf{u}^{q}=(\mathbf{u}_{1},\cdots, \mathbf{u}_{q})$. Then by Proposition \ref{Lagrange} we infer that $\mathbf{u}^{q}$ is a solution of the corresponding subsystem. By elliptic regularity we get that $u_{i}\in C^{2}(\overline{\Omega})$.
Observe that $u_{i}\equiv 0$ on $\partial\Omega$, $i\in I^{q}:=I_1\cup\ldots
\cup I_q$. Taking $x_{1}\in \partial\Omega$, there exists $\rho_{1}>0$
such that
\begin{align}\label{3-p-23}
 \Theta^{p}&:= \max_{i\in I^{q}}\sup_{x\in B_{2\rho_{1}}(x_{1})\cap \Omega}u_{i}^{p}(x)\leq \frac{1}{\displaystyle \max_{(i,j)\in \mathcal{K}_{2}}\{|\beta_{ij}|\}}\min\Bigg\{\frac{SC_1}{2d^3 } \frac{N\min\{l_1,\ldots,l_m\}}{4d^3}, ~~ \frac{1}{d\alpha_1^{p}\alpha_3^{p-2}}\Bigg\},
\end{align}
where $C_1$ is defined in Lemma \ref{Bounded}, and
\begin{equation}\label{eq:t_and_theta}
\alpha_1= \left(\frac{N^2\sum_{h=1}^ml_h}{(N-2)\theta}\right)^{\frac{1}{2^*-2}}, ~~
\alpha_3= \min_{q+1\leq h \leq m}\left\{\left(\frac{1}{2}\min_{i\in I_h}\left\{1+\frac{\lambda_i}{\lambda_1(\Omega)}\right\}\right)^{\frac{1}{2^*-2}}\right\},
\end{equation}
where
\begin{equation}\label{def-5}
\theta:= \min\left\{\frac{SC_1}{2}, \frac{N}{4}l_1, \cdots, \frac{N}{4}l_m\right\}.
\end{equation}
Therefore, there exists $\rho>0$ and $x_0\in \Omega$ such that $B_{2\rho}(x_{0})\subseteq B_{2\rho_{1}}(x_{1})\cap \Omega$, and so
\begin{equation}\label{3-p-16}
\sup_{x\in B_{2\rho}(x_{0})}u_{i}^{p}(x)\leq \Theta^{p} ~\forall i\in I^{q}.
\end{equation}

Take $x_{q+1},\cdots, x_{m}\in B_{2\rho}(x_{0})$ and $\rho_{q+1},\cdots,\rho_{m}>0$ such that $B_{2\rho_{q+1}}(x_{q+1}), \cdots, B_{2\rho_{m}}(x_{m})\subset B_{2\rho}(x_{0})$ and $B_{2\rho_i}(x_i)\cap B_{2\rho_j}(x_j)=\emptyset$ for every $i\neq j$, $i,j\in \{q+1,\ldots, m\}$. Let $\xi_{h}\in C^{1}_{0}(B_{2\rho_{h}}(x_{h}))$ be a nonnegative function with $\xi_{h}\equiv 1$ for $|x-x_{h}|\leq \rho_{h}, 0\leq\xi_{h}\leq 1, h=q+1,\ldots,m$.
Define
\begin{equation*}\label{3-p-21}
\mathbf{v}_{h}^{\varepsilon}:=\xi_{h}\mathbf{V}^{\varepsilon}_{h},
\end{equation*}
where
\begin{equation*}
\mathbf{V}_{h}^{\varepsilon}:=  \varepsilon^{-\frac{N-2}{2}}X_0(f_{max}^h)^{-\frac{N-2}{4}}U_{1,0}\left(\frac{x-x_{h}}{\varepsilon}\right), \quad \mathbf{V}_{h}=X_0(f_{max}^h)^{-\frac{N-2}{4}}U_{1,0} ~\text{ a minimizer for $l_h$}.
\end{equation*}
Similarly to \eqref{2.14}-\eqref{2.13}, we have the following inequalities
\begin{equation}\label{2.12-1,2-m}
\int_{\Omega}|\nabla v_{i}^{\varepsilon}|^{2}= \int_{\mathbb{R}^{N}}|\nabla V_{i}|^{2}+O(\varepsilon^{N-2}),\text{ for every } i\in I_{h}, ~h=q+1,\ldots,m,
\end{equation}
\begin{equation}\label{2.13-1,2-m}
\int_{\Omega}|v_{i}^{\varepsilon}|^{p}|v_{j}^{\varepsilon}|^{p}= \int_{\mathbb{R}^{N}}| V_{i}|^{p}| V_{j}|^{p}+O(\varepsilon^{N}),\text{ for every } (i,j)\in I_{h}^{2}, ~h=q+1,\ldots,m.
\end{equation}
\begin{equation*}\label{2.14-1,2-m}
\int_{\Omega}|v_{i}^{\varepsilon}|^{2}\geq C_{i}\varepsilon^{2}+O(\varepsilon^{N-2}), \text{ for every } i\in I_{h}, ~h=q+1,\ldots,m,
\end{equation*}
where $C_{i}$ is defined in \eqref{Constant-2-9}.
Moreover, we have
\begin{align}\label{Coupled term-1}
\int_{\Omega}|v_{i}^{\varepsilon}|^{p} &\leq \int_{B_{2\rho_{h}}(x_{h})}|V_{i}^{\varepsilon}|^{\frac{N}{N-2}}
= C\int_{B_{2\rho_{h}}(0)} \left(\frac{\varepsilon}{\varepsilon^2+|x|^2}\right)^{\frac{N}{2}}\nonumber\\
& =C\int_{B_{\varepsilon}(0)}\left(\frac{\varepsilon}{\varepsilon^2+|x|^2}\right)^{\frac{N}{2}}
+C\int_{B_{2\rho_{h}}(0)\backslash B_{\varepsilon}(0) }\left(\frac{\varepsilon}{\varepsilon^2+|x|^2}\right)^{\frac{N}{2}}\nonumber\\
&=C\varepsilon^{\frac{N}{2}}\int_{B_{1}(0)}
\left(\frac{1}{1+|x|^2}\right)^{\frac{N}{2}}+C\varepsilon^{\frac{N}{2}}\int_{B_{\frac{2\rho_{h}}{\varepsilon}}(0)\backslash B_{1}(0) }\left(\frac{1}{1+|x|^2}\right)^{\frac{N}{2}}\nonumber\\
& \leq C\varepsilon^{\frac{N}{2}}\left(\ln \frac{2\rho_{h}}{\varepsilon}+1\right)= C_2o(\varepsilon^2),
\end{align}
where $C_2$ is dependent only on $\beta_{ii}>0$, $\beta_{ij}$ for $(i,j)\in \mathcal{K}_{1}$.
Hence,
\begin{equation}\label{P norm bound}
\int_{\Omega}|v_{i}^{\varepsilon}|^{p}\leq 1 ~\text{ for } \varepsilon \text{ small enough, } \text{ for every } i\in I_{h}, ~h=q+1,\ldots,m.
\end{equation}

\medbreak
Consider
\begin{align}\label{3-p-4}
\Phi(t_{1},\cdots,t_{m}):&=J(t_{1}\mathbf{u}_{1},\cdots, t_{q}\mathbf{u}_{q}, t_{q+1}\mathbf{v}_{q+1}^{\varepsilon},\cdots,t_{m}\mathbf{v}_{m}^{\varepsilon})\nonumber\\
& =\frac{1}{2}\sum_{k=1}^{q}t_{k}^2\|\mathbf{u}_{k}\|_{k}^{2}
+\sum_{h=q+1}^{m}\frac{1}{2}t_{h}^2\sum_{i\in I_{h}}\int_{\Omega}|\nabla v^{\varepsilon}_{i}|^{2}+\lambda_{i}| v^{\varepsilon}_{i}|^{2}-\frac{1}{2p}M_{B}(\mathbf{u}^{\varepsilon})\mathbf{t}^p\cdot\mathbf{t}^p,
\end{align}
where
\[
\mathbf{u}^{\varepsilon}:= (\mathbf{u}_{1}^{\varepsilon},\cdots, \mathbf{u}_{m}^{\varepsilon})=(\mathbf{u}^{q}, \mathbf{v}_{q+1}^{\varepsilon},\cdots, \mathbf{v}_{m}^{\varepsilon}) \text{ and }
\mathbf{t}^p:=(t_1^p,\ldots,t_m^p)
\]

Inspired by \cite[Lemma 3.8]{TS 2019} we show that $M_{B}(\mathbf{u}^{\varepsilon})$ is positive definite.
\begin{lemma}\label{Con-4-1}
The matrix $M_{B}(\mathbf{u}^{\varepsilon})$ is strictly diagonally dominant and $M_{B}(\mathbf{u}^{\varepsilon})$ is
positive definite for small $\varepsilon$. Moreover, $\kappa_{min}\geq \theta$, where $\kappa_{min}$ is the minimum eigenvalues of $M_{B}(\mathbf{u}^{\varepsilon})$ and $\theta$ is defined in \eqref{eq:t_and_theta}.
\end{lemma}
\begin{proof}[\bf{Proof}]
Note that $\mathbf{u}\in \mathcal{N}$ and $\beta_{ij}\leq 0 ~\forall (i,j)\in \mathcal{K}_{2}$. Then we see that
\[
\sum_{(i,j)\in I_{k}^{2}}\int_{\Omega}\beta_{ij}|u_{i}|^{p}|u_{j}|^{p} -\sum_{l=1,l\neq k}^{q}
\left|\sum_{(i,j)\in I_{k}\times I_{l}}\int_{\Omega}\beta_{ij} |u_{i}|^{p}|u_{j}|^{p}\right|\geq SC_{1}, ~k=1,\ldots, q,
\]
where $C_{1}$ is defined in Lemma \ref{Bounded}. Combining this with \eqref{3-p-23}, \eqref{def-5}, \eqref{3-p-16} and \eqref{P norm bound}, following the proof of \cite[Lemma 3.7]{TS 2019}, we know that $\kappa_{min}\geq \theta$.

\end{proof}

It follows from Lemma \ref{Con-4-1} that $\Phi(t_{1},\cdots,t_{m})$ has a global maximum. Next, we will show that the each component of the maximum points is positive due to the condition $-\infty<\beta_{ij}< 0, \forall (i,j)\in \mathcal{K}_{2}$.
\begin{lemma}\label{Con-4-2}
There exist $\alpha_1,\alpha_2,\alpha_3>0$ such that
\begin{equation}\label{eq:maximum_estimate}
\sup_{t_{1},\cdots,t_{m}>0} J\left(t_{1}\mathbf{u}_{1},\cdots, t_{q+1}\mathbf{v}_{q+1}^{\varepsilon},\cdots,t_{m}\mathbf{v}_{m}^{\varepsilon}\right)=
\mathop{\max_{\alpha_2<t_{1},\cdots,t_{q}<\alpha_1}}_{\alpha_3<t_{q+1},\cdots,t_{m}<\alpha_1} J\left(t_{1}\mathbf{u}_{1},\cdots, t_{q+1}\mathbf{v}_{q+1}^{\varepsilon},\cdots,t_{m}\mathbf{v}_{m}^{\varepsilon}\right),
\end{equation}
where $\alpha_1,\alpha_3$ are defined in \eqref{eq:t_and_theta}.
\end{lemma}
\begin{proof}[\bf{Proof}]
Recalling the definition of $\Phi$ (see \eqref{3-p-4}), it follows from Lemma \ref{Con-4-1} that
\begin{align}\label{3-p-4.2}
\Phi(t_{1},\cdots,t_{m})
 \leq\frac{1}{2}\sum_{k=1}^{q}t_{k}^2\|\mathbf{u}_{k}\|_{k}^{2}
+\sum_{h=q+1}^{m}\frac{1}{2}t_{h}^2\sum_{i\in I_{h}}\int_{\Omega}|\nabla v^{\varepsilon}_{i}|^{2}+\lambda_{i}| v^{\varepsilon}_{i}|^{2}\, dx-\frac{1}{2p}\theta\sum_{k=1}^{m}t^{2p}_{k}.
\end{align}
Similarly to \eqref{2.15} we infer that
\begin{align}\label{3-p-4.3}
\sum_{i\in I_{h}}\int_{\Omega}|\nabla v^{\varepsilon}_{i}|^{2}+\lambda_{i}| v^{\varepsilon}_{i}|^{2} \, dx< Nl_{h},~~ \forall h=q+1, \ldots, m,
\end{align}
for small $\varepsilon$. It follows from Theorem \ref{Energy Estimates} that
\begin{equation}\label{3-p-4.4}
\|\mathbf{u}_{k}\|_{k}^{2}\leq \sum_{h=1}^q\|\mathbf{u}_{h}\|_{h}^{2}=Nc_{1,\ldots,q}\leq N\sum_{h=1}^ml_h.
\end{equation}
We deduce from \eqref{3-p-4.2}, \eqref{3-p-4.3} and \eqref{3-p-4.4} that
\begin{align*}
\Phi(t_{1},\cdots,t_{m})\leq  \sum_{k=1}^q \frac{N}{2}t_k^2\sum_{h=1}^ml_h + \sum_{h=q+1}^m \frac{N}{2}t_k^2 l_k-  \sum_{k=1}^{m}\frac{1}{2^*}\theta t_k^{2^*}
				 \leq \sum_{k=1}^m \left(\frac{N}{2}\sum_{h=1}^ml_ht_k^2 -\frac{1}{2^*}\theta t_k^{2^*}\right),
\end{align*}
which yields that $\Phi(t_{1},\cdots,t_{m})<0$ when $t_{1},\ldots,t_{m}>\alpha_1$. Thus,
\begin{align}\label{1,2,3-7}
\max_{t_{1},\ldots,t_{m}>0}& J\left(t_{1}\mathbf{u}_{1},\cdots, t_{q}\mathbf{u}_{q},t_{q+1}\mathbf{v}_{q+1}^{\varepsilon},\cdots,t_{m}\mathbf{v}_{m}^{\varepsilon}\right)
 =\nonumber\\\
& \max_{0<t_{1},\ldots,t_{m} \leq \alpha_1}J\left(t_{1}\mathbf{u}_{1},\cdots, t_{q}\mathbf{u}_{q},t_{q+1}\mathbf{v}_{q+1}^{\varepsilon},\cdots,t_{m}\mathbf{v}_{m}^{\varepsilon}\right).
\end{align}
Since $v_i^\varepsilon\cdot v_j^\varepsilon=0$ whenever $i\in I_h,j\in I_l$, $h\neq l$, $h,l\in \{q+1,\ldots, m\}$, then we have that
\begin{align*}
J(&t_{1}\mathbf{u}_{1},\cdots, t_{q}\mathbf{u}_{q},t_{q+1}\mathbf{v}_{q+1}^{\varepsilon},\cdots,t_{m}\mathbf{v}_{m}^{\varepsilon}) =\sum_{k=1}^{q}\left(\frac{1}{2}t_{k}^2\|\mathbf{u}_{k}\|_{k}^{2}
-\frac{1}{2p}t_{k}^{2p}\sum_{(i,j)\in I_{k}^2}\int_{\Omega}\beta_{ij} |u_{i}|^{p}|u_{j}|^{p}\right) \nonumber\\
& \quad+\sum_{h=q+1}^{m}\left(t_{h}^2\|\mathbf{v}^{\varepsilon}_{h}\|_{h}^{2}-\frac{1}{2p}t_{h}^{2p}\sum_{(i,j)\in I_{h}^{2}}\int_{\Omega}\beta_{ij}| v^{\varepsilon}_{i}|^{p}|v^{\varepsilon}_{j}|^{p}\right)\nonumber\\
& \quad+\left(\sum_{k=1}^{q}\sum_{h=q+1}^{m}\frac{t_{k}^pt_{h}^p}{p}\sum_{(i,j)\in I_{k}\times I_{h}}\int_{\Omega}|\beta_{ij}| |u_{i}|^{p}|v^{\varepsilon}_{j}|^{p}+\sum_{k,l=1}^{q}\sum_{l\neq k}\frac{t_{k}^pt_{l}^p}{p}\sum_{(i,j)\in I_{k}\times I_{l}}\int_{\Omega}|\beta_{ij}| |u_{i}|^{p}|u_{j}|^{p}\right)\nonumber\\
& =: \sum_{k=1}^qf_k(t_{k})+ \sum_{h=q+1}^m g^\varepsilon_h(t_{h})+G_\varepsilon(t_{1},\ldots,t_{q},t_{q+1},\ldots,t_{m}).
\end{align*}
Note that $1<p<2$ and $2p=2^*$. Consider the function
\begin{equation*}\label{t-function}
f_k(t_{k})=\frac{1}{2}t_{k}^2\|\mathbf{u}_{k}\|_{k}^{2}
-\frac{1}{2^*}t_{k}^{2^*}\sum_{(i,j)\in I_{k}^2}\int_{\Omega}\beta_{ij} |u_{i}|^{p}|u_{j}|^{p}.
\end{equation*}
Therefore, for every $k=1,\ldots,q$, the function $f_k$ is increasing whenever
\[
0<t_k<\alpha_2:= \min_{1\leq k\leq q}\left\{\frac{\|\mathbf{u}_{k}\|_{k}^{2}}{\sum_{(i,j)\in I_{k}^2}\int_{\Omega}\beta_{ij} |u_{i}|^{p}|u_{j}|^{p}}\right\}>0.
\]
Thus,
\begin{equation}\label{Maximum-1}
\max_{0<t_1,\ldots,t_q}\sum_{k=1}^{q}f_k(t_k)=\max_{\alpha_2\leq t_1,\ldots,t_q}\sum_{k=1}^{q}f_k(t_k)
\end{equation}
Similarly, $g^\varepsilon_h$ is increasing whenever $0<t_h<t^\varepsilon_h$, where
\begin{equation}\label{Inequality-1.2}
(t^\varepsilon_h)^{2^*-2}=\frac{\sum_{i\in I_{h}}\int_{\Omega}|\nabla v^{\varepsilon}_{i}|^{2}+\lambda_{i}| v^{\varepsilon}_{i}|^{2}}{\sum_{(i,j)\in I_{h}^{2}}\int_{\Omega}\beta_{ij}| v^{\varepsilon}_{i}|^{p}|v^{\varepsilon}_{j}|^{p}}.
\end{equation}
Note that $\lambda_i\in (-\lambda_1(\Omega),0)$ for every $i=1,\ldots,d$, then we have
\[
\int_\Omega(|\nabla u|^2+\lambda_iu^2)\geq \left(1+\frac{\lambda_i}{\lambda_1(\Omega)}\right)\int_\Omega|\nabla u|^2 ~\forall u\in H_0^1(\Omega).
\]
Combining this with \eqref{2.12-1,2-m}, \eqref{2.13-1,2-m} and \eqref{Inequality-1.2} we know that
\begin{align*}
(t^\varepsilon_h)^{2^*-2}&\geq \frac{\min_{i\in I_h}\left\{1+\frac{\lambda_i}{\lambda_1(\Omega)}\right\}\sum_{i\in I_{h}}\int_{\mathbb{R}^{N}}|\nabla V_{i}|^{2}+O(\varepsilon^{N-2})}{\sum_{(i,j)\in I_{h}^{2}}\int_{\mathbb{R}^{N}}\beta_{ij}| V_{i}|^{p}| V_{j}|^{p}+O(\varepsilon^{N})}\nonumber\\
&\geq \frac{1}{2}\min_{i\in I_h}\left\{1+\frac{\lambda_i}{\lambda_1(\Omega)}\right\} ~~ \text{ for } \varepsilon \text{ small enough. }
\end{align*}
We have that $g^\varepsilon_h$ is increasing whenever $0<t_h<\alpha_3$ for every $h=q+1,\ldots,m$. Therefore,
\begin{equation}\label{Maximum-2}
\max_{0<t_{q+1},\ldots,t_m}\sum_{k=q+1}^{m}g^\varepsilon_h(t_h)=\max_{\alpha_3\leq t_{q+1},\ldots,t_m}\sum_{k=q+1}^{m}g^\varepsilon_h(t_h).
\end{equation}
Since $G_\varepsilon(t_{1},\ldots,t_{q},\ldots,t_{m})\geq 0$, combining this with \eqref{Maximum-2}, \eqref{1,2,3-7} and \eqref{Maximum-1}, we see that \eqref{eq:maximum_estimate} holds.
\end{proof}

Denote
\begin{equation*}
f(t_1,\ldots,t_q):=\frac{1}{2}\sum_{k=1}^{q}t_{k}^2\|\mathbf{u}_{k}\|_{k}^{2}
-\frac{1}{2p}\sum_{k,l=1}^{q}t_{k}^pt_{l}^p\sum_{(i,j)\in I_{k}\times I_{l}}\int_{\Omega}\beta_{ij} |u_{i}|^{p}|u_{j}|^{p}.
\end{equation*}

\begin{lemma}\label{Maximum-2}
If $c_{1,\ldots,q}$ is attained by $(\mathbf{u}_{1},\cdots,\mathbf{u}_{q})$, then
\begin{equation*}
c_{1,\ldots,q}=\max_{t_{1},\cdots,t_{q}>0}f(t_{1},\cdots,t_{q}).
\end{equation*}
\end{lemma}
\begin{proof}[\bf{Proof}]
We deduce from Proposition \ref{Lagrange} that $\mathbf{u}^q=(\mathbf{u}_{1},\cdots,\mathbf{u}_{q})$ is a nonnegative solution of the corresponding system, and so $(\mathbf{u}_{1},\cdots,\mathbf{u}_{q})\in \mathcal{N}_{1,\ldots,q}$. Hence,
\begin{equation*}
\sum_{l=1}^{q}\sum_{(i,j)\in I_{k}\times I_{l}}\int_{\Omega}\beta_{ij} |u_{i}|^{p}|u_{j}|^{p}=\|\mathbf{u}_k\|_k^2>0.
\end{equation*}
Note that $\beta_{ij}\leq0 ~\forall (i,j)\in \mathcal{K}_2$. Therefore, following the proof of Lemma 2.2 in \cite{Wu Y-Z 2017} we know that $(1,1,\ldots,1)$ is the unique critical point of $f$ and is a global maximum, that is
\[
\max_{t_{1},\cdots,t_{q}>0}f(t_{1},\cdots,t_{q})=f(1,1,\cdots,1)=c_{1,\ldots,q}.\qedhere
\]
\end{proof}

\begin{lemma}\label{Con-4-3}
We have
\begin{equation*}\label{eq:maximum_estimate-2}
\max_{t_{1},\cdots,t_{m}>0} J\left(t_{1}\mathbf{u}_{1},\cdots, t_{q}\mathbf{u}_{q},t_{q+1}\mathbf{v}_{q+1}^{\varepsilon},\cdots,t_{m}\mathbf{v}_{m}^{\varepsilon}\right)<
c_{1,\ldots,q}+\sum_{h=q+1}^{m}l_{h}-\delta
\end{equation*}
for small $\varepsilon$, where $\delta$ is as in Theorem \ref{Energy Estimates} and $c_{1,\ldots,q}$ is defined in \eqref{Minimizer-4}.
\end{lemma}
\begin{proof}[\bf{Proof}]
Note that $1<p<2$ and $2p=2^*$.
By \eqref{3-p-23}, \eqref{3-p-16} and \eqref{Coupled term-1},  for $k=1,\ldots, q$, $h=q+1,\ldots,m$ and $\alpha_3<t_{h}<\alpha_1$, $0<t_{k}<\alpha_1$, we have
\begin{align}\label{1,2-11-m}
\sum_{k=1}^{q}t_{k}^pt_{h}^p\left|\sum_{(i,j)\in I_{k}\times I_{h}}\beta_{ij}
\int_{\Omega}|u_{i}|^{p}|v^{\varepsilon}_{j}|^{p}\right| &\leq d\max_{(i,j)\in \mathcal{K}_{2}}|\beta_{ij}|\alpha_1^{p}\Theta^pt_{h}^{p-2}t_{h}^2\sum_{i\in I_{h}}\int_{B_{2\rho}(x_{0})}|v^{\varepsilon}_{i}|^{p}\nonumber\\
& \leq d\max_{(i,j)\in \mathcal{K}_{2}}|\beta_{ij}|\alpha_1^{p}\alpha_3^{p-2}\Theta^pt_{h}^2\sum_{i\in I_{h}}\int_{B_{2\rho}(x_{0})}|v^{\varepsilon}_{i}|^{p}\nonumber\\
& \leq C_2o(\varepsilon^2)t_{h}^2,
\end{align}
where $C_2$ is dependent only on $\beta_{ii}>0$, $\beta_{ij}\geq 0$ for $(i,j)\in \mathcal{K}_{1}$. Based upon \eqref{1,2-11-m} and Lemma \ref{Maximum-2}, the rest of proof is exactly as \cite[Lemma 3.10]{TS 2019}, so we omit it.
\end{proof}

\begin{proof}[\bf Proof of Theorem \ref{Energy Estimation-1,2,m-1}.]
Without loss of generality, we prove that $c<c_{1,\ldots,p}+\sum_{h=p+1}^{m}l_{h}-\delta$.
By Lemma \ref{Con-4-2}, we know that $\Phi$ has a global maximum $\mathbf{t}^{\varepsilon}$ in $\overline{\mathbb{R}^{m}_{+}}$, and
\begin{equation*}\label{t-positive}
t^{\varepsilon}_{k}\geq \alpha_2>0 ~~\forall k=1,\ldots,q, ~\text{ and }~
t^{\varepsilon}_{h}\geq \alpha_3>0 ~~\forall h=q+1,\ldots,m.
\end{equation*}
Thus, $\partial_{k}\Phi(\mathbf{t}^{\varepsilon})= 0$ for every $k=1,\ldots,m$, and so
\begin{equation}\label{Manifold-8}
\left(t_{1}^{\varepsilon}\mathbf{u}_{1},\cdots,t_{q}^{\varepsilon}\mathbf{u}_{q}, t_{q+1}^{\varepsilon}\mathbf{v}_{q+1}^{\varepsilon},\cdots,t_{m}^{\varepsilon}\mathbf{v}_{m}^{\varepsilon}\right)\in \mathcal{N}.
\end{equation}
We deduce from \eqref{Manifold-8} and Lemma \ref{Con-4-3} that
\begin{align*}
c &\leq J\left(t_{1}^{\varepsilon}\mathbf{u}_{1},\cdots,t_{q}^{\varepsilon}\mathbf{u}_{q},t_{q+1}^{\varepsilon}\mathbf{v}_{q+1}^{\varepsilon},\cdots,t_{m}^{\varepsilon}\mathbf{v}_{m}^{\varepsilon}\right)\\
&\leq
\max_{t_{1},\ldots,t_{m}>0}J(t_{1}\mathbf{u}_{1},\cdots,t_{q}\mathbf{u}_{q},t_{q+1}\mathbf{v}_{q+1}^{\varepsilon},\cdots,t_{m}\mathbf{v}_{m}^{\varepsilon}) <c_{1,\ldots,q}+\sum_{h=q+1}^{m}l_{h}-\delta. \qedhere
\end{align*}
\end{proof}

Inspired by \cite{TS 2019}, we prove Theorem \ref{Theorem-1} by induction in the number of sub-groups. We start in the next subsection by proving the first induction step.
\begin{proof}[\bf Conclusion of the proof  of Theorem \ref{Theorem-1}]
Based on Proposition \ref{PS Sequence}, Theorem \ref{Energy Estimates} and Theorem \ref{Energy Estimation-1,2,m-1}, the proof follows exactly as the general case of \cite[Theorem 1.1]{TS 2019}, so we omit it.
\end{proof}

\section{\bf Asymptotic behavior of nonnegative solutions and Proof of Theorem \ref{Sign-changing solutions}}\label{sec5}

In this section, we study the asymptotic behaviors of nonnegative solutions when $\beta_{ij}\rightarrow 0^-$ or $\beta_{ij}\rightarrow -\infty$ for any $(i,j)\in \mathcal{K}_2$. This is a cornerstone in the proofs of  existence of least energy positive solutions in Theorem \ref{thm:leps}-(3),(4), which we show in the next section. After obtaining the precise asymptotic behaviors of nonnegative solutions as $\beta_{ij}\rightarrow -\infty$ for any $(i,j)\in \mathcal{K}_2$, we present the proof of Theorem \ref{Sign-changing solutions}. Recall the definitions of $J_\Gamma, \mathcal{N}_\Gamma, c_\Gamma$, all depending also on $\beta_{ij}$.
\smallbreak

\noindent \textbf{Notation.} In this section we will take sequences $\beta_{ij}^n$ for $(i,j)\in \mathcal{K}_2$ and we use the notation $J_\Gamma^n, \mathcal{N}_\Gamma^n, c_\Gamma^n$ for the corresponding functional, Nehari manifold and energy level, respectively. As usual, we also set $J_\Gamma^n=J^n, \mathcal{N}_\Gamma^n=\mathcal{N}^n$, $c_\Gamma^n=c^n$ whenever $\Gamma=\{1,\ldots,m\}$.

\subsection{Asymptotic behavior of nonnegative solutions when $\beta_{ij}\rightarrow 0^-$, $(i,j)\in \mathcal{K}_2$}
In this subsection, we investigate the asymptotic behavior of nonnegative solutions when $\beta_{ij}\rightarrow 0^-$ for any $(i,j)\in \mathcal{K}_2$.
Given $h=1,\ldots, m$, consider the following system
\begin{equation}\label{System-h-1}
\begin{cases}
-\Delta u_{i}+\lambda_{i}u_{i}=\sum_{j \in I_{h}}\beta_{ij}|u_{j}|^{p}|u_{i}|^{p-2}u_{i}  \text{ in } \Omega,\\
u_{i}\in H^{1}_{0}(\Omega)  \quad \forall i\in I_{h}.
\end{cases}
\end{equation}
We recall that, with $\Gamma=\{h\}$ (and dropping the parenthesis in the notations from now on), the ground-state level is
\begin{equation*}\label{Energy-h}
c_{h}:= \inf_{\mathbf{u}\in \mathcal{N}_{h}}J_{h}(\mathbf{u}), \quad \text{ where } \quad J_{h}(\mathbf{u}):=\int_{\Omega}\frac{1}{2}\sum_{i\in I_{h}}(|\nabla u_{i}|^{2}+\lambda_{i}u_{i}^{2})-\frac{1}{2p}\sum_{(i,j)\in I_{h}^{2}}
\beta_{ij}|u_{j}|^{p}|u_{i}|^{p},
\end{equation*}
where
\begin{equation*}
\mathcal{N}_{h}:=\Big\{\mathbf{u}\in (H^{1}_{0}(\Omega))^{a_{h}-a_{h-1}}:\mathbf{u}\neq \mathbf{0} \text{ and }
\langle \nabla J_{h}(\mathbf{u}),\mathbf{u}\rangle=0 \Big\}.
\end{equation*}
It is easy to see that $c_{h}>0$, and
\begin{equation*}
c_{h}= \inf_{\mathbf{u}\neq \mathbf{0}}\max_{t>0}J_{h}(t\mathbf{u})=\inf_{\mathbf{u}\neq \mathbf{0}}\frac{1}{N}
\left(\frac{\sum_{i\in I_h}\int_\Omega |\nabla v_i|^2+\lambda_i|v_i|^2}{\left(\sum_{(i,j)\in I_{h}^{2}}\int_{\Omega}\beta_{ij}|v_{j}|^{p}|v_{i}|^{p}\right)^{\frac{2}{2^*}}}\right)^{\frac{N}{2}}.
\end{equation*}
and we get the following
\begin{equation}\label{Vector Sobolev Inequality-2-4}
\sum_{(i,j)\in I_{h}^{2}}\int_{\Omega}\beta_{ij}|v_{j}|^{p}|v_{i}|^{p}\leq \left(Nc_h\right)^{\frac{2}{2-N}} \left(\sum_{i\in I_h}\int_\Omega |\nabla v_i|^2+\lambda_i|v_i|^2\right)^{\frac{2^*}{2}} ~\forall v_i\in H^1_0(\Omega), i\in I_h.
\end{equation}
(note the paralelism with $l_h$, related to the same subsystem in the whole space) from Theorem \ref{Theorem-1} we know that a nonnegative ground-state exists under \eqref{Conditions for the same group}, while Theorem \ref{Energy Estimates} implies that  $0<c_h<l_h$.
Then we have the following theorem
\begin{thm}\label{asymtotic behavior-1}
Let $\mathbf{a}$ be an m-decomposition of $d$ for some $1\leq m\leq d$ and assume conditions \eqref{eq:coefficients}--\eqref{eq:generalcondition_beta} hold true.  Let $\{\beta_{ij}^n\}_{(i,j)\in \mathcal{K}_2}$ be a sequence with $\beta_{ij}^n\rightarrow 0^-$ as $n\rightarrow \infty$, and $(\mathbf{u}_1^n, \ldots, \mathbf{u}_m^n)$ is a nonnegative solution  of \eqref{S-system}  with $\beta_{ij}=\beta_{ij}^n, (i,j)\in \mathcal{K}_2$ (obtained by Theorem \ref{Theorem-1}). Then, passing to a subsequence, for every $h=1,\ldots,m$ we have
\begin{equation*}
u_i^n \rightarrow \overline{u}_i \text{ strongly in } H_0^1(\Omega) \text{ as } n\rightarrow \infty, i\in I_h,
\end{equation*}
where $\overline{\mathbf{u}}_h\neq \mathbf{0}$ is a ground state of \eqref{System-h-1}. Moreover,

\begin{equation}\label{limit energy-3}
\lim_{n\rightarrow\infty}\int_{\Omega}\beta_{ij}^n|u_i^n|^p |u_j^n|^p= 0 \text{ for every } (i,j)\in\mathcal{K}_2, \quad \text{ and } \quad \overline{c}:=\lim_{n\rightarrow\infty}c^n=\sum_{h=1}^m c_h.
\end{equation}
\end{thm}

Before we prove the above theorem, we present an accurate upper bound estimate for the limit of $c^n$.
\begin{lemma}\label{Energy-3}
Suppose that the conditions in Theorem \ref{asymtotic behavior-1} are satisfied. Then
\begin{equation*}
\limsup_{n\rightarrow\infty}c^n\leq \sum_{h=1}^m c_h.
\end{equation*}
\end{lemma}
\begin{proof}[\bf{Proof}]
For each $h=1,\ldots, m$, let $\mathbf{w}_h$ be a ground state solution of \eqref{System-h-1}. We claim that, if $n$ is sufficiently large, then
\begin{equation}\label{solutions to algebric}
\text{there exist } t_1^n,\cdots, t_m^n>0 \text{ such that } \left(t_1^n\mathbf{w}_1,\cdots, t_m^n\mathbf{w}_m\right)\in \mathcal{N}^n.
\end{equation}
This is equivalent to
\begin{equation}\label{algebric system}
(t_h^n)^2\|\mathbf{w}_h\|_h^2=(t_h^n)^{2p}\sum_{(i,j)\in I_h^2}\int_\Omega\beta_{ij}|w_i|^p|w_j|^p+\sum_{k\neq h}(t_h^n)^p(t_k^n)^p\sum_{(i,j)\in I_h\times I_k}\int_\Omega\beta_{ij}^n|w_i|^p|w_j|^p,
\end{equation}
where $h=1,\ldots,m$.

For every $h=1,\ldots,m$, consider
\begin{equation*}
\widetilde{f}_h(t_1,\cdots,t_m,\mathbf{s}):=t_h^{2p}\sum_{(i,j)\in I_h^2}\int_\Omega\beta_{ij}|w_i|^p|\omega_j|^p+\sum_{k\neq h}t_h^pt_k^p\sum_{(i,j)\in I_h\times I_k}\int_\Omega s_{ij} |w_i|^p|w_j|^p-t_h^2\|\mathbf{w}_h\|_h^2,
\end{equation*}
where $t_1,\ldots, t_m>0$, $\mathbf{s}=(s_{ij})_{(i,j)\in \mathcal{K}_2}\in\mathbb{R}^{|\mathcal{K}_2|}$. Obviously, for every $h=1,\ldots,m$, $\widetilde{f}_h(1,\cdots,1,\mathbf{0})=0$. Note that for every $h=1,\ldots,m$
\begin{equation*}
\frac{\partial \widetilde{f}_h}{\partial t_h}(1,\cdots,1,\mathbf{0})=(2p-2)\|\mathbf{w}_h\|_h^2, \quad
\frac{\partial \widetilde{f}_h}{\partial t_k}(1,\cdots,1,\mathbf{0})=0, ~h\neq k.
\end{equation*}
Define the matrix
\begin{equation*}
G:= \left(\frac{\partial \widetilde{f}_h}{\partial t_k}(1,\cdots,1,\mathbf{0})\right)_{1\leq h,k\leq m}
\end{equation*}
then $det(G)>0$. Therefore, by the implicit function theorem, there exist continuous function $t_h(\mathbf{s}),h=1,\ldots,m$ and $\delta_3>0$ small such that for every $h=1,\ldots,m$, $\widetilde{f}_h(t_1(\mathbf{s}),\cdots,t_m(\mathbf{s}),\mathbf{s})=0$ for all $s_{ij}\in (-\delta_3,\delta_3)$. Moreover, $t_h(\mathbf{0})=1, h=1,\ldots, m$. Then by continuity there exists $\delta_4<\delta_3$ such that for every $h=1,\ldots,m$ we have
\begin{equation}\label{t-Bounded}
0<t_h(\mathbf{s})\leq 2 ~~\text{ for all } s_{ij}\in [-\delta_4,\delta_4].
\end{equation}
Since $\beta_{ij}^n\rightarrow 0$ as $n\rightarrow\infty$, then we see that $\beta_{ij}^n\in [-\delta_4,\delta_4]$ for any $(i,j)\in \mathcal{K}_2$ when $n$ is sufficiently large. Hence, \eqref{algebric system} has a solution $t_1^n,\cdots, t_m^n>0$ when $n$ is sufficiently large, where $t_h^n:= t_h(\mathbf{s}^n)$, with $s_{ij}^n:=\beta_{ij}^n$ for $(i,j)\in \mathcal{K}_2$.
Therefore, \eqref{solutions to algebric} is true. It follows from \eqref{t-Bounded} that $t_1^n,\ldots,t_m^n$ are uniformly bounded, and so by \eqref{algebric system} we see that
\begin{equation*}\label{solutions to algebric-2}
\lim_{n\rightarrow \infty}|t_h^n-1|=0, ~h=1,\ldots,m.
\end{equation*}
Combining this with \eqref{solutions to algebric} we have
\begin{equation*}
\limsup_{n\rightarrow \infty} c^n\leq \limsup_{n\rightarrow \infty}J^n\left(t_1^n\mathbf{w}_1,\cdots, t_m^n\mathbf{w}_m\right)= \sum_{h=1}^m\|\mathbf{w}_h\|_h^2=\sum_{h=1}^m c_h. \qedhere
\end{equation*}
\end{proof}
\begin{proof}[\bf Proof of Theorem \ref{asymtotic behavior-1}]
Let $\mathbf{u}_n=(\mathbf{u}_1^n,\ldots, \mathbf{u}_m^n)$ be a nonnegative solution with $\beta_{ij}=\beta_{ij}^n$ for every $(i,j)\in\mathcal{K}_2$, and $\beta_{ij}^n\rightarrow 0^-$ as $n\rightarrow \infty$. By Lemma \ref{Energy-3} we see that $\{u_i^n\}$ is uniformly bounded in $H^1_0(\Omega)$. Passing to subsequence, for $h=1,\ldots,m$ we may assume that
\begin{equation}\label{S-C-1-4}
u_{i}^{n}\to \overline{u}_i \quad \text{ weakly in } H^1_0(\Omega), \text{ strongly in } L^{2}(\Omega) \text{ and a.e. in }\Omega,   ~i\in I_{h}.
\end{equation}
Hence, $\overline{u}_i\geq 0$. Next, arguing by contradiction, we show that $\overline{\mathbf{u}}_h\neq \mathbf{0}$ for every $h=1,\ldots,m$.

{\bf Case 1}. All limiting components are zero: $\overline{\mathbf{u}}=(\overline{\mathbf{u}}_{1},\cdots, \overline{\mathbf{u}}_{h},\cdots, \overline{\mathbf{u}}_{m})\equiv \mathbf{0}$.

Denote
\begin{equation}\label{Limit-2-1}
{D}_h:=\lim_{n\rightarrow \infty}\sum_{i\in I_{h}}\int_{\Omega}|\nabla u_{i}^n|^{2}.
\end{equation}
We have ${D}_h>0$ for every $h=1,\ldots,m$, otherwise
\begin{equation*}
u_{i}^{n}\rightarrow 0 \text{ strongly in } ~H^1_0(\Omega), ~i\in I_{h},\quad \text{ hence} \quad \lim_{n\rightarrow \infty}\sum_{i\in I_h}|u_i^n|^2_{2p}=0,
\end{equation*}
which contradicts Lemma \ref{Bounded}.
Next, since $\mathbf{u}^n\in\mathcal{N}^n$ and $\beta_{ij}^n\leq 0$ for every $(i,j)\in\mathcal{K}_2$, we see that
\begin{align}
\sum_{i\in I_h}\int_\Omega|\nabla u_i^n|^2+\lambda_i|u_i^n|^2 &= \sum_{(i,j)\in I_h^2}\int_\Omega\beta_{ij}|u_{i}^n|^p|u_{j}^n|^p+\sum_{k\neq h}\sum_{(i,j)\in I_h\times I_k}\int_\Omega\beta_{ij}^n|u_{i}^n|^p|u_{j}^n|^p\\
& \leq \sum_{(i,j)\in I_h^2}\int_\Omega\beta_{ij}|u_{i}^n|^p|u_{j}^n|^p
\leq \left(Nl_h\right)^{\frac{2}{2-N}} \left(\sum_{i\in I_{h}}\int_{\Omega}|\nabla u_{i}^n|^{2}\right)^{\frac{2^*}{2}}.\label{Limit-1-1}
\end{align}
Note that $\int_\Omega|u_i^n|^2\rightarrow 0, i\in I_{h}$. Combining this with \eqref{Limit-2-1}-\eqref{Limit-1-1} we have
\begin{equation*}\label{Limit-3}
{D}_h\geq Nl_h.
\end{equation*}
It follows from Theorem \ref{Energy Estimates} and Lemma \ref{Energy-3} that
\begin{equation*}\label{Limit-4}
\sum_{h=1}^m l_h>\sum_{h=1}^m c_h\geq \lim_{n\rightarrow \infty}c^{n}=\frac{1}{N}\lim_{n\rightarrow \infty}\sum_{h=1}^m \|\mathbf{u}^n_h\|_h^2=\frac{1}{N}\sum_{h=1}^m {D}_h\geq \sum_{h=1}^m l_h,
\end{equation*}
which is a contradiction. Hence, Case 1 is impossible.

{\bf Case 2}. Only one component of $\overline{\mathbf{u}}$ is not zero.

Without loss of generality, we may assume that $\overline{\mathbf{u}}_{1}\neq \mathbf{0}, \overline{\mathbf{u}}_{2}=\cdots=\overline{\mathbf{u}}_{h}=\cdots= \overline{\mathbf{u}}_{m}=\mathbf{0}$. Similarly to Case 1, we know that
\begin{equation}\label{Limit-5-4}
{D}_h\geq Nl_h\quad \text{ for every } h=2,\ldots,m.
\end{equation}
Note that $\beta_{ij}^n\leq 0$ for every $(i,j)\in\mathcal{K}_2$. Testing the equation for $u_i^n$ in \eqref{S-system} by $\overline{u}_i, i\in I_1$ and integrating over $\Omega$, we see that
\begin{align}
\int_\Omega (\nabla u_i^n\nabla \overline{u}_i+\lambda_i u_i^n \overline{u}_i) & =\sum_{j\in I_1}\int_\Omega\beta_{ij}|u_j^n|^p|u_i^n|^{p-2}u_i^n\overline{u}_i+\sum_{j\not\in I_1}\int_\Omega\beta_{ij}^n|u_j^n|^p|u_i^n|^{p-2}u_i^n\overline{u}_i\nonumber\\
&\leq\sum_{j\in I_1}\int_\Omega\beta_{ij}|u_j^n|^p|u_i^n|^{p-2}u_i^n\overline{u}_i. \label{Limit-6-4}
\end{align}
By \eqref{S-C-1-4} we see that $|u_j^n|^p|u_i^n|^{p-2}u_i^n\rightharpoonup |\overline{u}_j|^p|\overline{u}_i|^{p-2}\overline{u}_i $ weakly in  $L^{\frac{2^*}{2^*-1}}(\Omega)$, therefore
\begin{equation*}\label{Limit-7-4}
\int_\Omega\beta_{ij}|u_j^n|^p|u_i^n|^{p-2}u_i^n\overline{u}_i\rightarrow \int_\Omega\beta_{ij}|\overline{u}_j|^p|\overline{u}_i|^{p}, ~\text{ as } n\rightarrow \infty.
\end{equation*}
So, passing to the limit in \eqref{Limit-6-4}, we find
\begin{equation}\label{Limit-10}
\int_\Omega (|\nabla \overline{u}_i|^2+\lambda_i|\overline{u}_i|^2)\leq \sum_{j\in I_1} \int_\Omega\beta_{ij}|\overline{u}_j|^p|\overline{u}_i|^{p} \qquad \forall i\in I_1,
\end{equation}
which, when combined with \eqref{Vector Sobolev Inequality-2-4}, yields
\begin{equation}\label{Limit-11-4}
Nc_1\leq \sum_{i\in I_1}\int_\Omega (|\nabla \overline{u}_i|^2+\lambda_i|\overline{u}_i|^2)\leq \sum_{(i,j)\in I_1^2} \int_\Omega\beta_{ij}|\overline{u}_j|^p|\overline{u}_i|^{p}.
\end{equation}
By \eqref{Limit-5-4}, Theorem \ref{Energy Estimates}  and Lemma \ref{Energy-3} we have
\begin{align*}
c_1+\sum_{h=2}^m l_h >\sum_{h=1}^m c_h &\geq \lim_{n\rightarrow \infty}c^n=\frac{1}{N}\lim_{n\rightarrow \infty}\|\mathbf{u}_h^n\|_h^2   \\
   &= \frac{1}{N}\sum_{i\in I_1}\int_\Omega (|\nabla \overline{u}_i|^2+\lambda_i|\overline{u}_i|^2)+\frac{1}{N}\sum_{h=2}^m {D}_h+\frac{1}{N}\lim_{n\rightarrow \infty}\sum_{i\in I_1}\int_\Omega |\nabla (u_i^n-\overline{u}_i)|^2  \geq c_1+\sum_{h=2}^m l_h,
\end{align*}
a contradiction. Thus, Case 2 is not true.

{\bf Case 3}. Only $q$ components of $\overline{\mathbf{u}}$ are not zero, for some $q\in \{2,3,\ldots, m-1\}$.

Without loss of generality, we assume that $\overline{\mathbf{u}}_{1}\neq \mathbf{0}, \overline{\mathbf{u}}_{2}\neq \mathbf{0}, \cdots, \overline{\mathbf{u}}_{q}\neq \mathbf{0}, \overline{\mathbf{u}}_{q+1}=\cdots= \mathbf{u}_{h}=\cdots= \overline{\mathbf{u}}_{m}=\mathbf{0}$. Similarly to Case 1, we know that
\begin{equation}\label{Limit-5-2-4}
{D}_h\geq Nl_h \quad \text{ for every } h=q+1,\ldots,m.
\end{equation}
Similarly to \eqref{Limit-11-4}, for $h=1,\ldots,q$ we get that
\begin{equation}\label{Limit-13-4}
\sum_{i\in I_h}\int_\Omega (|\nabla \overline{u}_i|^2+\lambda_i|\overline{u}_i|^2)\geq Nc_h.
\end{equation}
We deduce from \eqref{Limit-5-2-4}, \eqref{Limit-13-4}, Theorem \ref{Energy Estimates}  and Lemma \ref{Energy-3} that
\begin{align*}
\sum_{h=1}^q c_h+\sum_{h=q+1}^m l_h>\sum_{h=1}^m c_h &\geq \lim_{n\rightarrow \infty}c^n =\frac{1}{N}\lim_{n\rightarrow \infty}\|\mathbf{u}_h^n\|_h^2\\
  & \geq \frac{1}{N}\sum_{h=1}^q\|\overline{\mathbf{u}}_h\|_h^2+\frac{1}{N}\sum_{h=q+1}^m {D}_h
   \geq \sum_{h=1}^q c_h+\sum_{h=q+1}^m l_h,
\end{align*}
a contradiction. Therefore, Case 3 is impossible.

Since Case 1, Case 2 and Case 3 are impossible, then we get that $\overline{\mathbf{u}}_{h}\neq \mathbf{0}$ for any $h=1,2,\ldots, m$. From the previous arguments, we also know that
\begin{equation*}\label{Limit-15-4}
\|\overline{\mathbf{u}}_h\|_h^2=\sum_{i\in I_h}\int_\Omega |\nabla \overline{u}_i|^2+\lambda_i|\overline{u}_i|^2\geq Nc_h\quad \text{ for every } h=1,\ldots,m.
\end{equation*}
Combining this with Lemma \ref{Energy-3} we have
\begin{align*}
\sum_{h=1}^m c_h &\geq \lim_{n\rightarrow \infty}c^n=\frac{1}{N}\lim_{n\rightarrow \infty}\|\mathbf{u}_h^n\|_h^2   \\
   &= \frac{1}{N}\sum_{h=1}^m\sum_{i\in I_h}\int_\Omega |\nabla \overline{u}_i|^2+\lambda_i|\overline{u}_i|^2+\lim_{n\rightarrow \infty}\frac{1}{N}\sum_{h=1}^m\sum_{i\in I_h}\int_\Omega |\nabla (u_i^n-\overline{u}_i)|^2\geq \sum_{h=1}^m c_h,\\
\end{align*}
which yields that
\begin{equation*}
\lim_{n\rightarrow \infty}\sum_{i\in I_h}\int_\Omega |\nabla (u_i^n-\overline{u}_i)|^2=0,\quad
\|\overline{\mathbf{u}}_h\|_h^2=Nc_h, ~h=1,\ldots,m.
\end{equation*}
Hence $u_i^n \rightarrow \overline{u}_i$  strongly in  $H_0^1(\Omega)$ as  $n\rightarrow \infty$, $i\in I_h$, $h=1,\ldots,m$, $\overline{\mathbf{u}}_h\neq \mathbf{0}$ is a ground state of \eqref{System-h-1}, and \eqref{limit energy-3} holds true.
\end{proof}

\subsection{Asymptotic behavior of nonnegative solutions when $\beta_{ij}\rightarrow -\infty$, $(i,j)\in \mathcal{K}_2$: phase separation}

In this subsection, we study the asymptotic behavior of nonnegative solutions when $\beta_{ij}\rightarrow -\infty$ for any $(i,j)\in \mathcal{K}_2$, proving Theorem \ref{Phase Separation} and also Corollary \ref{Phase Separation2}. To start with, we need an estimate in the spirit of Lemma \ref{Energy-3}. However, unlike the case $\beta_{ij}\to 0^-$, here we need to estimate also the energy level sets of subsystems. Having that in mind define, for $\Gamma\subseteq \{1,\ldots, m\}$, let
\begin{equation*}\label{limiting files minimizing}
c^{\infty}_\Gamma:=\inf_{\mathcal{N}_\Gamma^\infty}J_\Gamma^{\infty},\quad \text{ with }\quad J_\Gamma^{\infty}(\mathbf{u}):=\frac{1}{N}\sum_{h\in\Gamma}\|\mathbf{u}_h\|_h^2,
\end{equation*}
and
\begin{equation*}
 \mathcal{N}_\Gamma^{\infty}:=\left\{\mathbf{u}=\{\mathbf{u}_h\}_{h\in \Gamma}: \ 0<\|\mathbf{u}_h\|_h^2\leq \sum_{(i,j)\in I_h^2}\int_{\Omega}\beta_{ij}|u_{i}|^p|u_{j}|^p\ \forall  h\in \Gamma \text{ and }
u_{i}u_{j}\equiv 0\ \forall (i,j)\in \mathcal{K}_2 \right\}.
\end{equation*}
Set $c_\Gamma^n=c^n ,c_\Gamma^{\infty}=c^{\infty}, J_\Gamma^{\infty}=J^{\infty}$ when $\Gamma=\{1,\ldots,m\}$. It is easy to see that $c_\Gamma^{\infty}$ is well defined due to the fact that the set $\mathcal{N}_\Gamma^{\infty}$ is non empty.
\begin{lemma}\label{relationship}
Assume the conditions in Theorem \ref{Phase Separation} hold. Then we get that
\begin{equation*}
\limsup_{n\rightarrow \infty}c_\Gamma^{n}\leq c_\Gamma^{\infty}.
\end{equation*}
\end{lemma}
\begin{proof}[\bf{Proof}]
Given $\mathbf{u}\in \mathcal{N}_\Gamma^{\infty}$ such that $\int_{\Omega}|u_{i}|^p|u_{j}|^p=0$ for every $i\in I_h, ~j\in I_k, h\neq k \in \Gamma$, then we have $(t_1\mathbf{u}_1, \cdots,t_h\mathbf{u}_h, \cdots, t_m\mathbf{u}_m)\in\mathcal{N}_\Gamma^{n}$, where
\begin{equation*}\label{t-def}
t_h^{2p-2}= \frac{\|\mathbf{u}_h\|_h^2}{\sum_{(i,j)\in I_h^2}\int_{\Omega}\beta_{ij}|u_{i}|^p|u_{j}|^p}\leq 1.
\end{equation*}
Hence,
\begin{equation*}
c_\Gamma^{n}\leq \frac{1}{N}\sum_{h\in\Gamma}t_h^2\|\mathbf{u}_h\|_h^2\leq \frac{1}{N}\sum_{h\in\Gamma}\|\mathbf{u}_h\|_h^2=J_\Gamma^{\infty}(\mathbf{u}),
\end{equation*}
which yields that $\limsup_{n\rightarrow \infty}c_\Gamma^{n}\leq J_\Gamma^{\infty}(\mathbf{u})$ and thus $\limsup_{n\rightarrow \infty}c_\Gamma^{n}\leq c_\Gamma^{\infty}$.
\end{proof}
\begin{proof}[\bf Proof of Theorem \ref{Phase Separation}]
Let $\mathbf{u}_n=(\mathbf{u}_1^n,\ldots, \mathbf{u}_m^n)$ be a nonnegative solution associated with $c^n$, with $\beta_{ij}=\beta_{ij}^n$ for every $(i,j)\in\mathcal{K}_2$, and $\beta_{ij}^n\rightarrow -\infty$ as $n\rightarrow \infty$ (for $N\geq 5$, such a solution exists by Theorem \ref{Theorem-1}, while for $N=4$ it exists by \cite[Theorem 1.1]{TS 2019} - taking also into account Remark \ref{rem:differencesN}). By Lemma \ref{relationship} we see that $\{u_i^n\}$ is uniformly bounded in $H^1_0(\Omega)$. Passing to subsequence, for $h=1,\ldots,m$ we may assume that
\begin{equation*}
u_{i}^{n}\to \overline{u}_i \quad \text{ weakly in } H^1_0(\Omega), \text{ strongly in } L^{2}(\Omega) \text{ and a.e. in }\Omega,   ~i\in I_{h}.
\end{equation*}
We now follow closely the proof of Theorem \ref{asymtotic behavior-1}, showing by a contradiction argument that $\overline{\mathbf{u}}_h\neq \mathbf{0}$ for every $h=1,\ldots,m$. In fact, we see that Cases 1 and 2 therein (i.e, when all limiting components except possibly one vanish)  lead to a contradiction in the same way. So we focus on the case:
\begin{equation}\label{quote}
\text{Only $q$ components of $\mathbf{u}^\infty$ are not zero, for some $q\in \{2,3,\ldots, m-1\}$.}
\end{equation}
Without loss of generality, we may assume that $\mathbf{u}_{1}^\infty\neq \mathbf{0}, \mathbf{u}^\infty_{2}\neq \mathbf{0}, \cdots, \mathbf{u}_{q}^\infty\neq \mathbf{0}, \mathbf{u}_{q+1}^\infty=\cdots= \mathbf{u}_{m}^\infty=\mathbf{0}$. Similarly to Case 1 in Theorem \ref{asymtotic behavior-1}, we know that
\begin{equation}\label{Limit-5-2}
D_h:=\lim_{n\rightarrow \infty}\sum_{i\in I_{h}}\int_{\Omega}|\nabla u_{i}^n|^{2}\geq Nl_h,~h=q+1,\ldots,m.
\end{equation}
Notice that
\begin{equation*}
\sum_{(i,j)\in I_h^2} \int_\Omega\beta_{ij}|u_j^n|^p|u_i^n|^{p}+\sum_{k\neq h}\sum_{(i,j)\in I_h\times I_k} \int_\Omega\beta_{ij}^n|u_j^n|^p|u_i^n|^{p}=\|\mathbf{u}_h^n\|_h^2>0,
\end{equation*}
which implies that, for any $(i,j)\in \mathcal{K}_2$,
\begin{equation*}
|\beta_{ij}^n|\int_\Omega|u_j^n|^p|u_i^n|^{p}\leq\sum_{(i,j)\in I_h\times I_k} |\beta_{ij}^n|\int_\Omega|u_j^n|^p|u_i^n|^{p}\leq\sum_{(i,j)\in I_h^2} \int_\Omega\beta_{ij}|u_j^n|^p|u_i^n|^{p}\leq C.
\end{equation*}
Thus $u_j^\infty\cdot u_i^\infty\equiv 0$ for every $(i,j)\in \mathcal{K}_2$, since $\beta_{ij}\to -\infty$.

Similarly to \eqref{Limit-10}, for $h=1,\ldots,q$ we get that
\begin{equation*}\label{Limit-13}
\sum_{i\in I_h}\int_\Omega |\nabla u_i^\infty|^2+\lambda_i|u_i^\infty|^2\leq \sum_{(i,j)\in I_h^2} \int_\Omega\beta_{ij}|u_j^\infty|^p|u_i^\infty|^{p}.
\end{equation*}
Hence, $(\mathbf{u}_{1}^\infty,\cdots,\mathbf{u}_{q}^\infty)\in \mathcal{N}_\Gamma^{\infty}$ for $\Gamma=\{1,\ldots,q\}$. Then
\begin{equation}\label{Limit-14}
J^\infty_{1,\ldots,q}(\mathbf{u}_{1}^\infty,\cdots,\mathbf{u}_{q}^\infty)=\frac{1}{N}\sum_{h=1}^q\|\mathbf{u}^\infty_h\|_h^2\geq c^\infty_{1,\ldots,q}.
\end{equation}
We deduce from \eqref{Limit-5-2}, \eqref{Limit-14}, Lemma \ref{relationship} and Theorem \ref{Energy Estimation-1,2,m-1} that
\begin{align*}
c^\infty_{1,\ldots,q}+ \sum_{h=q+1}^m l_h -\delta &\geq \lim_{n\rightarrow \infty}c^n_{1,\ldots,q}+ \sum_{h=q+1}^m l_h -\delta\geq \lim_{n\rightarrow \infty}c^n =\frac{1}{N}\lim_{n\rightarrow \infty}\|\mathbf{u}_h^n\|_h^2    \\
   &\geq \frac{1}{N}\sum_{h=1}^q\|\mathbf{u}^\infty_h\|_h^2+\frac{1}{N}\sum_{h=q+1}^m D_h\geq c^\infty_{1,\ldots,q}+\sum_{h=q+1}^m l_h,
\end{align*}
a contradiction. Therefore also \eqref{quote} is impossible and we get that $\mathbf{u}^\infty_{h}\neq \mathbf{0}$ for any $h=1,2,\ldots, m$.

\smallbreak

From the previous arguments we know that $\mathbf{u}^\infty \in \mathcal{N}_\Gamma^{\infty}$. Reasoning as in the last part of the proof of Theorem \ref{asymtotic behavior-1}, it is now straightforward to prove  that $u_i^n\rightarrow u_i^\infty$ strongly in $H^1_0(\Omega)$, $i\in I_h, h=1\ldots,m$, that $\mathbf{u}^\infty$ achieves $c^\infty$ and
 \begin{equation}\label{limit energy}
\lim_{n\rightarrow\infty}c^n=c^\infty=J(\mathbf{u}^\infty),\quad  \lim_{n\rightarrow\infty}\int_{\Omega}\beta_{ij}^n|u_i^n|^p |u_j^n|^p= 0 \text{ for every } (i,j)\in\mathcal{K}_2.
\end{equation}
Finally, following the proof of \cite[Lemma 6.1]{Zou 2012}, we can get that $\{u_i^n\}$ is uniformly bounded in $L^\infty(\Omega)$. Then by \cite[Theorem 1.5]{STTZ 2016} we know that $u_i^\infty$ is Lipschitz continuous in $\Omega$, $i\in I_h, h=1\ldots,m$, that the convergence is in H\"older spaces and that  \eqref{limit equation-4} is true, and $\overline{\Omega}=\bigcup_{h=1}^m\overline{\left\{|\mathbf{u}_h^\infty| >0\right\}}$.
\end{proof}

\begin{proof}[\bf Proof of Corollary \ref{Phase Separation2}.]
Based on Theorem \ref{Phase Separation}, taking $m=d$, $\mathbf{a}=(1,\ldots, m)$, we see that $u_i^\infty\in C^{0,1}(\overline{\Omega})$, $i=1,\ldots,d$ and $\overline{\Omega}=\bigcup_{i=1}^d\overline{\{u_i^\infty>0\}}$. Observe that for every $i=1,\ldots,d$, the set $\{u_i^\infty>0\}$ is open, and $\{u_i^\infty>0\}\cap\{u_j^\infty>0\}=\varnothing$ when $i\neq j$. For every $i=1,\ldots,d$, let $v_i$ be a least energy positive solution of
\begin{equation*}
-\Delta u +\lambda_i u=\beta_{ii}u^{2^*-1}, \quad u\in H_0^1(\{u_i^\infty>0\}).
\end{equation*}
Thus $J(0,\cdots,v_i,\cdots,0)\leq J(0,\cdots,u_i^\infty,\cdots,0)$. Note that $\{v_i>0\}\cap\{v_j>0\}=\varnothing$ when $i\neq j$, and $(v_1,\cdots,v_d)\in\mathcal{N}^n$.
Then, we deduce that
\begin{align*}
 \lim_{n\rightarrow\infty}c^n  &\leq J(v_1,\cdots,v_i,\cdots,v_d) =\sum_{i=1}^d J(0,\cdots,v_i,\cdots,0) \leq \sum_{i=1}^d J(0,\cdots,u_i^\infty,\cdots,0)= \lim_{n\rightarrow\infty}c^n.
\end{align*}
It follows that
$J(0,\cdots,v_i,\cdots,0)= J(0,\cdots,u_i^\infty,\cdots,0)$, and so $u_i^\infty$ is a least energy positive solution of
\begin{equation*}
-\Delta u +\lambda_i u=\beta_{ii}u^{2^*-1}, \quad u\in H_0^1(\{u_i^\infty>0\}).
\end{equation*}

It remains to show that for every $i=1,\ldots,d$, the set $\{u_i^\infty>0\}$ is connected. By contradiction, without loss of generality we assume that $\{u_1^\infty>0\}$ has at least two connected components $\widetilde{\Omega}_{1}$ and $\widetilde{\Omega}_{2}$. Denote $\widetilde{u}_{i}:= \chi_{\widetilde{\Omega}_{i}}u_1^\infty$, $i=1,2$, where $\chi_{\widetilde{\Omega}_{i}}(x)$ is the indicator function of the set $\widetilde{\Omega}_{i}$.
Obviously, $\widetilde{u}_{i}\in H_0^1(\{u_1^\infty>0\})$ and $\widetilde{u}_{i} \not\equiv 0$. Note that $J'(u_1^\infty,0,\cdots,0)(\widetilde{u}_{1},0,\cdots,0)=0$, $i=1,2$, then we get that
\begin{equation*}
\int_{\{u_1^\infty>0\}}|\nabla \widetilde{u}_{i}|^2+\lambda_1|\widetilde{u}_{i}|^2=\int_{\{u_1^\infty>0\}}\beta_{11}|\widetilde{u}_{i}|^{2^*},
\end{equation*}
which yields that $J(\widetilde{u}_{i},0,\cdots,0)\geq J(u_1^\infty,0,\cdots,0)$. Therefore we get a contradiction from the fact that $J(u_1^\infty,0,\cdots,0)\geq J(\widetilde{u}_{1},0,\cdots,0)+J(\widetilde{u}_{2},0,\cdots,0)\geq 2J(u_1^\infty,0,\cdots,0)$.
\end{proof}

\subsection{Proof of Theorem \ref{Sign-changing solutions}}
In this subsection, we prove Theorem \ref{Sign-changing solutions} by using Theorem \ref{Phase Separation}.
\begin{proof}[\bf Proof of Theorem \ref{Sign-changing solutions}.]
Set $m=d=2$, $\lambda_1=\lambda_2=\lambda$ and $\beta_{11}=\beta_{22}=\mu$. It is well known that the solutions of \eqref{B-N problem} are critical points of the following functional $I: H_0^1(\Omega)\rightarrow \mathbb{R}$
\begin{equation*}
I(u)=\frac{1}{2}\int_\Omega|\nabla u|^2+\lambda |u|^2-\frac{1}{2^*}\int_\Omega\mu|u|^{2^*}.
\end{equation*}
Define
\begin{equation*}
\mathcal{N}_1:=\left\{u\neq 0: \int_\Omega|\nabla u|^2+\lambda |u|^2=\int_\Omega\mu|u|^{2^*}\right\},\quad
\mathcal{M}:=\left\{ u\in H_0^1(\Omega): u^{\pm} \neq 0, u^+ \in \mathcal{N}_1 \text{ and } u^- \in \mathcal{N}_1 \right\},
\end{equation*}
where $u^+:=\max\{u,0\}$ and $u^-:=\max\{-u,0\}$. Obviously, any sign-changing solutions of \eqref{B-N problem} belong to $\mathcal{M}$.
Consider
\begin{equation*}
a:= \inf_{u\in \mathcal{M}}I(u).
\end{equation*}
We claim that $a=I(u_1^\infty-u_2^\infty)$. In fact, we deduce from the proof of Theorem \ref{Phase Separation} that $I(u_1^\infty-u_2^\infty)=J(u_1^\infty, u_2^\infty)$ and $u_1^\infty-u_2^\infty\in \mathcal{M}$. Therefore,
\begin{equation}\label{limit-energy-4-1}
 I(u_1^\infty-u_2^\infty)\geq a.
\end{equation}
On the other hand, for any $u\in \mathcal{M}$, we see that $(u^+,u^-)\in \mathcal{N}^n$. Thus, $c^n\leq J(u^+,u^-)$. Combining these with \eqref{limit energy} we get that
\begin{equation*}\label{limit-energy-4}
I(u_1^\infty-u_2^\infty)=J(u_1^\infty, u_2^\infty)=\lim_{n\rightarrow\infty}c^n\leq J(u^+,u^-)=I(u),
\end{equation*}
which implies that
\begin{equation}\label{limit-energy-4-2}
I(u_1^\infty-u_2^\infty)\leq a.
\end{equation}
Thus, by \eqref{limit-energy-4-1} and \eqref{limit-energy-4-2} we have $a=I(u_1^\infty-u_2^\infty)$ and $u_1^\infty- u_2^\infty\in \mathcal{M}$. By a standard argument (see for instance \cite[Lemma 5.2]{Zou 2015}), we get that $I'(u_1^\infty-u_2^\infty)=0$. Therefore, $u_1^\infty-u_2^\infty$ is a least energy sign-changing solution to \eqref{B-N problem}, and has two nodal domains.
\end{proof}

\section{Existence of least energy positive solutions: proof of Theorem \ref{thm:leps}}\label{sec6}

\begin{proof}[\bf Proof of Theorem \ref{thm:leps}-(1)]
It follows from Theorem \ref{Theorem-1} that $\widetilde{c}$ is achieved by a ground-state solution $\mathbf{u}\in \widetilde{\mathcal{N}}$, where $\widetilde{c}$ and $\widetilde{\mathcal{N}}$ are defined in \eqref{ground state}. For $d=2$, the proof of Theorem \ref{thm:leps}-(1)  follows by \cite[Theorem 1.3]{Zou 2015}, or by reasoning as in  \cite[page 385]{TO 2016}. In the general case of $d>2$ equations, the result can be shown following word by word the reasoning of the proof of \cite[Theorem 1.2]{TO 2016} (see page 386 therein), arguing by induction in the number of equations: assuming that all $c_\Gamma$ for every $\emptyset \neq \Gamma \subsetneq\{1,\ldots, d\} $ are associated with a fully-nontrivial solution, and that  $c=c_{\{1,\ldots, n\}}$ is associated with a ground-state of type $(u_1,\ldots, u_{d-1},0)$ with $u_i\neq 0$, then one finds $w\in H^1_0(\Omega)$, $\theta>0$ and $t>0$ such that $(tu_1,\ldots, tu_{d-1},t\theta w)\in \mathcal{N}$ and $J(tu_1,\ldots, tu_{d-1},t\theta w)<J(u_1,\ldots, u_d)$, a contradiction. We stress that it is of key importance in these arguments the fact that $1<p<2$.
\end{proof}

\begin{proof}[\bf Proof of Theorem \ref{thm:leps}-(2)] This is a direct consequence of Theorem \ref{Theorem-1} for the case of an $m=d$ decomposition $\mathbf{a}=(1,\ldots, d)$.
\end{proof}

\begin{proof}[\bf Proof of Theorem \ref{thm:leps}-(3)]
\smallbreak

\emph{Step1.} Assume by contradiction that the statement doesn't hold. Then there exist sequences $\beta^n_{ij}\rightarrow 0^-$  for every $(i,j)\in \mathcal{K}_2$ as $n\rightarrow \infty$  such that  $\mathbf{u}_n=(\mathbf{u}_1^n,\ldots, \mathbf{u}_m^n)$ has a zero component. Here $\mathbf{u}_n$ denotes the solution of \eqref{S-system} with level $c^n=J^n(\mathbf{u}^n)$, provided by Theorem \ref{Theorem-1}.  We also denote by $\mathcal{N}^n$ the associated Nehari-type set.

Without loss of generality and passing if necessary to a subsequence, we can assume that $|I_m|>1$ and that there exists $m_{\overline{i}}\in I_m$  such that $u^n_{m_{\overline i}}\equiv 0$ for every $n$. By Lemma \ref{Bounded} we see (passing if necessary again to a subsequence) that there exists $m_{\widehat{i}}\in I_m$, such that
\begin{equation}\label{Positive-3-4-1}
|u^n_{m_{\widehat{i}}}|^2_{2p}\geq C_1/ d  ~\text{ for every } n.
\end{equation}
To fix ideas we suppose that $m_{\widehat{i}}<m_{\overline{i}}$. Hence, we assume that $\mathbf{u}_m^n=(u^n_{a_{m-1}+1},\cdots,u^n_{m_{\widehat{i}}},\cdots,u^n_{m_{\overline{i}}},\cdots,u^n_{a_m})$, where $u^n_{m_{\overline{i}}}\equiv 0$. Note that $u_i^n$ are uniformly bounded in $H^1_0(\Omega)$.
Moreover, we deduce from Theorem \ref{asymtotic behavior-1} that
\begin{equation}\label{Positive-3-4}
u_i^n\rightarrow \overline{u}_i \text{ strongly in } H^1_0(\Omega) \text{ as } n\rightarrow \infty\ \forall i,\qquad \lim_{n\rightarrow\infty}\int_{\Omega}\beta_{ij}^n|u_i^n|^p |u_j^n|^p= 0 \text{ for every } (i,j)\in\mathcal{K}_2.
\end{equation}
The plan now is to replace the zero component in $\mathbf{u}_m^n$ by a multiple of $u_{m_{\widehat{i}}}$, showing that this decreases the energy, leading to a contradiction. Having this in mind, let
\begin{equation*}\label{De-u-4}
\widehat{\mathbf{u}}_m^n(s):=(u^n_{a_{m-1}+1},\cdots,u^n_{m_{\widehat{i}}},\cdots,|s|^{\frac{1}{p}-1}su^n_{m_{\widehat{i}}},\cdots,u^n_{a_m})
~~\forall s\in \R.
\end{equation*}
We wish to show that, for $s>0$, there exists $t_1,\ldots, t_m>0$ such that $(t_1^n\mathbf{u}_1,\ldots,t_{m-1}^n \mathbf{u}_{m-1}, t_{m}^n \widehat{\mathbf{u}}_m^n)\in\mathcal{N}^n$, which is equivalent to
\begin{align*}
f_h^n(t_1,\cdots,t_m,s) &:= A_{hh}^n t_h^{2p}+\sum_{k\neq h}^{m}A_{hk}^nt_k^pt_h^p+A_{hm_{\overline{i}}}^nt_m^pt_h^ps-\|\mathbf{u}^n_h\|_h^2t_h^2=0,\ \text{ for every $h=1,\ldots,m-1$,}\\
f_m^n(t_1,\cdots,t_m,s) &:=\left[A^n_{mm}+2sA^n_{mm_{\overline{i}}}+s^2\int_{\Omega}\beta_{m_{\overline{i}}m_{\overline{i}}}\left|u^n_{m_{\widehat{i}}}\right|^{2p}
\right]t_m^{2p}+\sum_{k=1}^{m-1}A_{km}^nt_k^pt_m^p\nonumber\\
& +\sum_{k=1}^{m-1}A_{km_{\overline{i}}}^nt_k^pt_m^ps-\left(\|\mathbf{u}^n_m\|^2_{m}
+|s|^{\frac{2}{p}-1}s\|u^n_{m_{\widehat{i}}}\|^2_{m_{\overline{i}}}\right)t_m^2=0,
\end{align*}
where
\begin{equation*}
A_{hk}^n:=\sum_{(i,j)\in I_h\times I_k}\int_{\Omega}\beta_{ij}^n|u_i^n|^p |u_j^n|^p,\qquad A_h^n:=\sum_{i\in I_h}\int_{\Omega}\beta_{im_{\overline{i}}}^n|u_i^n|^p |u^n_{m_{\widehat{i}}}|^p.
\end{equation*}
For future use, we point out that, due to the assumption $\beta_{ij}^n\to 0$ for every $(i,j)\in \mathcal{K}_2$ and by the uniform $H^1_0$--bounds of $\mathbf{u}^n$, we have
\begin{equation}\label{limit-1-4}
\lim_{n\rightarrow \infty}A^n_{hk}=0, \ \forall h\neq k,\quad \lim_{n\rightarrow \infty}A^n_{hm_{\overline{i}}}=0 \ \forall h=1,\ldots, m-1.
\end{equation}

\smallbreak

\emph{Step 2. }
Obviously, $f_h^n\in C^1(\mathbb{R}^{m+1}, \mathbb{R})$ for every $h=1,\ldots,m$. Set $\mathbf{a}=(1,\cdots,1,0)$. Note that $\mathbf{u}^n\in \mathcal{N}^n$, then $f_h^n(\mathbf{a})=0$ for every $h=1,\ldots,m$, that is, $\|\mathbf{u}^n_h\|_h^2=A_{hh}^n+\sum_{k\neq h}^{m}A_{hk}^n$.  Define
\begin{equation*}\label{Positive-8-1-4}
B_1:= \left(\frac{\partial f_h^n}{\partial t_k}(\mathbf{a})\right)_{1\leq h,k\leq m}.
\end{equation*}
Next, we show that $det(B_1)\geq C$, where $C$ is independent on $n$. By a direct computation, for $h,k=1,\ldots,m-1$ we have
\begin{equation}\label{Positive-6-4}
\frac{\partial f_h^n}{\partial t_h}(\mathbf{a})=2pA_{hh}^n+p\sum_{k\neq h}^{m}A_{hk}^n-2\|\mathbf{u}^n_h\|_h^2,\quad \frac{\partial f_h^n}{\partial t_k}(\mathbf{a})=pA_{hk}^n, ~k\neq h, \quad  \frac{\partial f_h^n}{\partial s}(\mathbf{a})=A_{hm_{\overline{i}}}^n.
\end{equation}
Similarly we see that
\begin{equation}\label{Positive-8-4}
\frac{\partial f_m^n}{\partial t_m}(\mathbf{a})=2pA^n_{mm}+p\sum_{k=1}^{m-1}A_{km}^n-2\|\mathbf{u}^n_m\|_m^2,\quad
\frac{\partial f_m^n}{\partial t_k}(\mathbf{a})=pA_{km}^n ,~~k\neq m,\quad \frac{\partial f_m^n}{\partial s}(\mathbf{a})=2A^n_{mm_{\overline{i}}}+\sum_{k=1}^{m-1}A_{km_{\overline{i}}}^n.
\end{equation}
We see from \eqref{Positive-6-4}, \eqref{Conditions for the same group}, $\beta^n_{ij}\leq0 ~(i,j)\in \mathcal{K}_2$, $\mathbf{u}^n\in \mathcal{N}^n$ that
\begin{align*}
\frac{\partial f_h^n}{\partial t_h}(\mathbf{a})-\sum_{k\neq h}^{m}\left|\frac{\partial f_h^n}{\partial t_k}(\mathbf{a})\right|&=2pA_{hh}^n+2p\sum_{k\neq h}^{m}A_{hk}^n-2\|\mathbf{u}^n_h\|_h^2. \\
		&\geq (2p-2)\|\mathbf{u}^n_h\|_h^2   \geq 2(p-1)S\sum_{i\in I_h}|u_i^n|^2_{2p}\geq 2(p-1)SC_1  ~~h=1,\ldots,m,
\end{align*}
where $C_1$ is defined in Lemma \ref{Bounded} independent from $n$. Therefore, $B_1$ is strictly diagonally dominant.
Moreover, by the Gershgorin circle theorem we have $det(B_1)\geq C$,

\smallbreak

\emph{Step 3. }The implicit function theorem implies the existence of functions $t_h^n(s)$ of class $C^1$ in some interval $(-\sigma_n, \sigma_n)$ for $\sigma_n>0$, satisfying $t_h^n(0)=1$  and
\begin{equation}\label{Positive-13-4}
f_h^n(t_1^n(s),\cdots,t_h^n(s),\cdots,t_m^n(s),s)\equiv 0, ~s\in [0, \sigma_n), \text{ for every } h=1,\ldots,m,
\end{equation}
that is, $(t_1^n(s)\mathbf{u}_1^n,\cdots,t_{m-1}^n(s)\mathbf{u}_{m-1}^n,t_m^n(s)\widehat{\mathbf{u}}_m^n(s))\in \mathcal{N}^n$. Therefore,
\begin{equation}\label{Positive-13-4}
\sum_{k=1}^m\frac{\partial f_h^n}{\partial t_k}(\mathbf{a})\frac{\partial t_k^n}{\partial s}(0)+\frac{\partial f_h^n}{\partial s}(\mathbf{a})=0  ~\text{ for every } h=1,\ldots,m,
\end{equation}
where (by Cramer's rule)
\begin{equation}\label{Positive-14-4}
\frac{\partial t_k^n}{\partial s}(0):= \lim_{s\rightarrow 0^+}\frac{t_k^n(s)-1}{s}=\frac{det(B_1^k)}{det(B_1)}  ~\text{ for every } k=1,\ldots,m,
\end{equation}
where
 \begin{equation*}\label{Positive-15-4}
B_1^k:=
\left(
\begin{matrix}
 \frac{\partial f_1^n}{\partial t_1}(\mathbf{a})&\cdots& \frac{\partial f_1^n}{\partial t_{k-1}}(\mathbf{a})&\frac{\partial f_1^n}{\partial s}(\mathbf{a})&\frac{\partial f_1^n}{\partial t_{k+1}}(\mathbf{a})&\cdots&\frac{\partial f_1^n}{\partial t_m}(\mathbf{a})\\
 \frac{\partial f_2^n}{\partial t_1}(\mathbf{a})&\cdots& \frac{\partial f_2^n}{\partial t_{k-1}}(\mathbf{a})&\frac{\partial f_2^n}{\partial s}(\mathbf{a})&\frac{\partial f_2^n}{\partial t_{k+1}}(\mathbf{a})&\cdots&\frac{\partial f_2^n}{\partial t_m}(\mathbf{a})\\
\cdots&\cdots&\cdots&\cdots&\cdots&\cdots&\cdots\\
\frac{\partial f_m^n}{\partial t_1}(\mathbf{a})&\cdots& \frac{\partial f_m^n}{\partial t_{k-1}}(\mathbf{a})&\frac{\partial f_m^n}{\partial s}(\mathbf{a})&\frac{\partial f_m^n}{\partial t_{k+1}}(\mathbf{a})&\cdots&\frac{\partial f_m^n}{\partial t_m}(\mathbf{a})\\
\end{matrix}
\right).
\end{equation*}
Observe that $\{u_i^n\}_{i\in I_h, h=1,\ldots,m}$ are uniformly bounded in $H_0^1(\Omega)$ (independently on $\beta_{ij}$ for $(i,j)\in \mathcal{K}_2$). Moreover, again from \eqref{Conditions for the same group} and $\beta^n_{ij}\leq0 ~(i,j)\in \mathcal{K}_2$ and the fact that $\mathbf{u}^n\in \mathcal{N}^n$,
\begin{equation*}\label{Positive-16-4}
A_{hh}^n,\ \sum_{k\neq h}^{m}|A_{hk}^n| \leq\|\mathbf{u}^n_h\|_h^2+A_{hh}^n\leq C,
\end{equation*}
where $C>0$ doesn't depend on $n$.
 Combining this with \eqref{Positive-6-4} and \eqref{Positive-8-4}, we get that
\begin{equation}\label{Positive-17-4}
\left|\frac{\partial f_h^n}{\partial t_k}(\mathbf{a})\right|,\ \left|\frac{\partial f_h^n}{\partial s}(\mathbf{a})\right| \leq C, \text{ for every } h,k=1,\ldots,m.
\end{equation}
It follows from  \eqref{Positive-14-4}, \eqref{Positive-17-4} the fact that $\det(B_1)\geq C$ that
\begin{equation*}\label{Positive-19-4}
\left|\frac{\partial t_k^n}{\partial s}(0)\right|\leq C, \text{ for every } k=1,\ldots,m,
\end{equation*}
where $C>0$ doesn't depend on $n$.

\smallbreak

\emph{Step 4.} Therefore, we deduce from \eqref{Positive-3-4}, \eqref{limit-1-4}, \eqref{Positive-6-4}, \eqref{Positive-13-4}, that
\begin{equation*}\label{Positive-20-4-4}
\left(2p\sum_{(i,j)\in I_h^2}\int_\Omega\beta_{ij}|\overline{u}_i|^p|\overline{u}_j|^p-2\|\overline{\mathbf{u}}_h\|_h^2\right)\lim_{n\rightarrow \infty}\frac{\partial t_h^n}{\partial s}(0)=0, ~h=1,\ldots, m-1.
\end{equation*}
On the other hand, by Theorem \ref{asymtotic behavior-1} we see that
\begin{equation*}\label{Limit-Positive}
2p\sum_{(i,j)\in I_h^2}\int_\Omega\beta_{ij}|\overline{u}_i|^p|\overline{u}_j|^p-2\|\overline{\mathbf{u}}_h\|_h^2=
(2p-2)\|\overline{\mathbf{u}}_h\|_h^2>0 ~\text{ for every } h=1,\ldots,m.
\end{equation*}
thus
\begin{equation}\label{Positive-20-4}
\lim_{n\rightarrow \infty}\frac{\partial t_h^n}{\partial s}(0)=0, \text{ for every } h=1,\ldots,m-1.
\end{equation}
Similarly to \eqref{Positive-20-4}, we have
\begin{equation}\label{Positive-20-2-4-1}
\lim_{n\rightarrow \infty}\frac{\partial t_m^n}{\partial s}(0)=-\frac{2\overline{A}_{m}}{2p\overline{A}_{mm}-2\|\overline{\mathbf{u}}_m\|_{m}^2}
=-\frac{2\overline{A}_{m}}{(2p-2)\|\overline{\mathbf{u}}_m\|_{m}^2},
\end{equation}
where
\begin{equation*}
\overline{A}_{mm}:= \sum_{(i,j)\in I_m^2}\int_{\Omega}\beta_{ij}|\overline{u}_i|^p |\overline{u}_j|^p, \quad \overline{A}_{m}:= \sum_{i\in I_m}\int_{\Omega}\beta_{im_{\overline{i}}}|\overline{u}_i|^p |\overline{u}_{m_{\widehat{i}}}|^p>0,
\end{equation*}
where the positivity of $\overline{A}_{m}$ is a consequence of \eqref{Positive-3-4-1} . We deduce from \eqref{Positive-20-2-4-1} that
\begin{equation}\label{Positive-20-2-4}
\lim_{n\rightarrow \infty}\frac{\partial t_m^n}{\partial s}(0)=:-L<0.
\end{equation}
Note that $t_{h}^n(s)=1+\frac{\partial t_h^n}{\partial s}(0)s+o_n(s)$ where, for each $n$ fixed, $o_n(s)/s\to 0$ as $s\to 0$. Then
\begin{equation*}\label{Positive-21-4}
\left|t_{h}^n(s)\right|^2=1+2\frac{\partial t_h^n}{\partial s}(0)s+o_n(s^2).
\end{equation*}
Combining this with the fact that $(t_1^n(s)\mathbf{u}_1^n,\cdots,t_{m-1}^n(s)\mathbf{u}_{m-1}^n,t_m^n(s)\widehat{\mathbf{u}}_m^n(s))\in \mathcal{N}^n$, we see that
\begin{align*}
c^n &\leq \frac{1}{N}\sum_{h=1}^{m-1}|t_{h}^n(s)|^2\|\mathbf{u}^n_h\|_h^2 + \frac{1}{N}|t_{m}^n(s)|^2\|\widehat{\mathbf{u}}_m^n\|_m^2\\
& = \frac{1}{N}\sum_{h=1}^{m-1}\left(1+2\frac{\partial t_h^n}{\partial s}(0)s+o_n(s^2)\right)\|\mathbf{u}^n_h\|_h^2+\frac{1}{N}\left(1+2\frac{\partial t_m^n}{\partial s}(0)s+o_n(s^2)\right)\left(\|\mathbf{u}^n_m\|^2_{m}+|s|^{\frac{2}{p}}\|u^n_{m_{\widehat{i}}}\|^2_{m_{\overline{i}}}\right)\\
& = \frac{1}{N}\sum_{h=1}^{m}\|\mathbf{u}^n_h\|_h^2+\frac{2}{N}\sum_{h=1}^{m}\frac{\partial t_h^n}{\partial s}(0)\|\mathbf{u}^n_h\|_h^2s+\frac{1}{N}\|u^n_{m_{\widehat{i}}}\|^2_{m_{\overline{i}}}|s|^{\frac{2}{p}}+o_n(s^2)\\
& = c^n+s \left(-\frac{2L}{N}\|\mathbf{u}^n_m\|_m^2+\frac{2}{N}\left(\frac{\partial t_m^n}{\partial s}(0)+L\right) \|\mathbf{u}^n_m\|_m^2+\frac{2}{N}\sum_{h=1}^{m-1}\frac{\partial t_h^n}{\partial s}(0)\|\mathbf{u}^n_h\|_h^2+\frac{1}{N}\|u^n_{m_{\widehat{i}}}\|^2_{m_{\overline{i}}}|s|^{\frac{2}{p}-2}s+o_n(s)\right).
\end{align*}
Therefore, by fixing first  an $n$ sufficiently large so that
\[
-\frac{2L}{N}\|\mathbf{u}^n_m\|_m^2+\frac{2}{N}\left(\frac{\partial t_m^n}{\partial s}(0)+L\right) \|\mathbf{u}^n_m\|_m^2+\frac{2}{N}\sum_{h=1}^{m-1}\frac{\partial t_h^n}{\partial s}(0)\|\mathbf{u}^n_h\|_h^2<0;
\]
since $1<p<2$, we can then take $s\sim 0$ in such a way that $c^n<c^n$, a contradiction. This finishes the proof.
\end{proof}

\begin{proof}[\bf Conclusion of the proof of Theorem \ref{thm:leps}-(4)]
Assume, in view of a contradiction, the existence of $\beta^n_{ij}\rightarrow -\infty$  for every $(i,j)\in \mathcal{K}_2$ as $n\rightarrow \infty$  such that  $\mathbf{u}_n=(\mathbf{u}_1^n,\ldots, \mathbf{u}_m^n)$ has a zero component.  Exactly as in the proof of (3), we assume that $|I_m|>1$ and that there exists $m_{\overline{i}}\in I_m$  such that $u^n_{m_{\overline i}}\equiv 0$ for every $n$, and $m_{\widehat{i}}\in I_m$ smaller than $m_{\overline i}$, such that $|u^n_{m_{\widehat{i}}}|^2_{2p}\geq C_1/ d$.
We deduce from Theorem \ref{Phase Separation} that
\begin{equation*}\label{Positive-3}
u_i^n\rightarrow u_i^\infty \text{ strongly in } H^1_0(\Omega) \text{ as } n\rightarrow \infty, ~~i\in I_h, h=1,\ldots,m,
\end{equation*}
and
\begin{equation*}\label{Positive-4}
\lim_{n\rightarrow\infty}\int_{\Omega}\beta_{ij}^n|u_i^n|^p |u_j^n|^p= 0 \text{ for every } (i,j)\in\mathcal{K}_2.
\end{equation*}
Define
\begin{equation*}
\widehat{\mathbf{u}}_m^n(s):=(u^n_{a_{m-1}+1},\cdots,u^n_{m_{\widehat{i}}},\cdots,|s|^{\frac{1}{p}-1}su^n_{m_{\widehat{i}}},\cdots,u^n_{m_j})
~~\forall s\in \R.
\end{equation*}

Following the similar arguments in the proof of (3), we see that there exist $C^1$ functions $t_1^n(s),\ldots,t_m^n(s)$ such that
$(t_1^n(s)\mathbf{u}_1^n,\cdots,t_{m-1}^n(s)\mathbf{u}_{m-1}^n,t_m^n(s)\widehat{\mathbf{u}}_m^n)\in \mathcal{N}^n$ for $s\in [0, \delta_n)$, where $\delta_n>0$.
For $h=1,\ldots, m-1$ by the assumption in Theorem \ref{thm:leps}-(4) we have,
\begin{equation*}
\left|A^n_{hm_2}\right|\leq \left|\sum_{i\in I_h}\int_{\Omega}\beta_{im_{\overline{i}}}^n|u_i^n|^p |u^n_{m_{\widehat{i}}}|^p\right|\leq M\sum_{i\in I_h}\left|\int_{\Omega}\beta_{im_{\widehat{i}}}^n|u_i^n|^p |u^n_{m_{\widehat{i}}}|^p\right|.
\end{equation*}
Hence, similarly to \eqref{Positive-20-4} and \eqref{Positive-20-2-4} we see that
\begin{equation*}\label{Positive-20}
\lim_{n\rightarrow \infty}\frac{\partial t_h^n}{\partial s}(0)=0, \text{ for every } h=1,\ldots,m-1
~\text{ and }
\lim_{n\rightarrow \infty}\frac{\partial t_m^n}{\partial s}(0)<0,
\end{equation*}
and we reach a contradiction just like in the previous proof.
\end{proof}

\section{The case of two groups}\label{sec8}
This section is devoted to the proof of Theorem \ref{existence-positive-two groups2}.
Now, we study the case of $d=3$ equations
\begin{equation*}\label{two groups}
-\Delta u_{i}+\lambda_{i}u_{i}=u_{i}^{p-1}\sum_{j = 1}^{3}\beta_{ij}  u^p_{j}   ~\text{ in } \Omega,\ i=1,2,3,
\end{equation*}
and $m=2$ groups. To fix ideas, we consider the partition $I_1=\{1,2\}$, $I_2=\{3\}$.  Associated with this partition, we set $c$, $J$, $\mathcal{N}$ as defined in the introduction, see \eqref{Functional-1}, \eqref{Manifold-1}, \eqref{Minimizer-1}. Based on Theorem \ref{Theorem-1}, we know that $c$ is achieved by a nonnegative solution $(\widehat{u}_1,\widehat{u}_2,\widehat{u}_3) \in \mathcal{N}_{12,3}$, and $(\widehat{u}_1,\widehat{u}_2)\neq \mathbf{0}, \widehat{u}_3\neq 0$. In order to prove that $(\widehat{u}_1,\widehat{u}_2,\widehat{u}_3)$ is a least energy positive solution, our strategy is to show that
\begin{equation}\label{energy inequality-4-1}
c< \min\{d_{13}, d_{23}\},
\end{equation}
where $d_{i3}$ ($i=1,2$) is defined as the least energy positive level of
\begin{equation}\label{two equations}
-\Delta u +\lambda_i u= \beta_{ii} u^{2p-1}+\beta_{i3} v^pu^{p-1},\quad  -\Delta v +\lambda_3 v= \beta_{33} v^{2p-1}+\beta_{i3} u^pv^{p-1} ~\text{ in } \Omega.
\end{equation}
Observe that if \eqref{energy inequality-4-1} is true, then we can get that both $\widehat{u}_1\neq 0$ and $\widehat{u}_2\neq 0$, and so $(\widehat{u}_1,\widehat{u}_2,\widehat{u}_3)$ is a least energy positive solution.

By Theorem \ref{thm:leps}-(2) we see that $d_{i3}$ is attained for $i=1,2$. Let us denote by $(u_1,u_3)$ a solution associated with $d_{13}$, and by $(\tilde u_2,\tilde u_3)$ a critical point of $d_{23}$. By elliptic regularity theory, we see that all these functions are of class $C^2$ in $\overline \Omega$. Set $M_i, \widetilde M_i$ and $m_i,\widetilde m_i$, depending only on $\lambda_i,\beta_{ii}$, such that
\begin{equation}\label{Maximum-15}
m_i \leq  \int_\Omega |u_i|^{2p} \leq M_i,\qquad \widetilde m_i\leq \int_\Omega |\tilde u_i|^{2p}\leq \widetilde M_i
\end{equation}
(this is a consequence of the uniform bounds in Lemma \ref{Bounded} and Theorem \ref{Energy Estimates} - which are independent from $\beta_{13},\beta_{23}<0$, together with the Sobolev embedding $H^1_0\hookrightarrow L^{2p}(\Omega)$.)

By Theorem \ref{Energy Estimation-1,2,m-1} we have
\begin{equation}\label{Energy estimate-4}
d_{13}<\min\left\{d_1+ l_3-\delta, ~d_3+l_1-\delta \right\},\quad d_{23}<\min\left\{d_2+ l_3-\delta, ~d_3+l_2-\delta \right\},
\end{equation}
where $\delta$ is only dependent on $\lambda_i,\beta_{ii}$, $l_i$ is defined in \eqref{eq:Ground State-2} with $|I_h|=1$, and $d_i$ is the ground state level of
\begin{equation*}
-\Delta u +\lambda_i u= \beta_{ii} u^{2p-1}, ~u\in H^1_0(\Omega), ~i=1,2,3.
\end{equation*}

\begin{lemma}\label{uniform bounded-2}
Assume that $(u_1,u_3)$ is a least energy positive solution of \eqref{two equations} with $i=1$, and $(\tilde u_2,\tilde u_3)$ is a least energy positive solution of \eqref{two equations} with $i=2$. Then there exist $C_{3},C_4, C_5, C_6>0$ such that for any $\beta_{13},\beta_{23}\in [-\sigma_1, -\sigma_0]$, we have
\begin{equation*}
\int_{\Omega}|\beta_{13}|u_1^pu_3^p\geq C_3, \quad
~\frac{\int_{\Omega}|\beta_{13}|u_1^pu_3^p}{\int_{\Omega}\beta_{11}u_1^{2p}}\geq C_4\quad \text{ and } \quad \int_{\Omega}|\beta_{23}|u_2^pu_3^p\geq C_5,\quad \frac{\int_{\Omega}|\beta_{23}|u_2^pu_3^p}{\int_{\Omega}\beta_{22}u_2^{2p}}\geq C_6.
\end{equation*}
where $C_3,C_4,C_5,C_6$ are only dependent on $\lambda_i,\beta_{ii},\sigma_0,\sigma_1$.
\end{lemma}
\begin{proof}[\bf Proof]
We prove the first two inequalities (which are related with system \eqref{two equations}  with $i=1$: the other two are completely analogous). By contradiction, take a sequence $\beta^n_{13} \in [-\sigma_1, -\sigma_0]$, $\beta^n_{13}\rightarrow \beta_0$ as $n\rightarrow \infty$, and let $(u_n,v_n)$ be a least energy positive solution when $\beta_{13}=\beta^n_{13}$, such that
\begin{equation}\label{Energy estimate-6-2}
\int_{\Omega}|\beta^n_{13}|u_n^pv_n^p\rightarrow 0  \text{ as } n\rightarrow \infty.
\end{equation}
Note that $d^n_{13}= \frac{1}{N}(\|u_n\|^2_1+\|v_n\|^2_3)$ and $l_i=\frac{1}{N}\beta_{ii}^{-\frac{N-2}{2}}\mathcal{S}^{\frac{N}{2}},i=1,3$, where $\mathcal{S}$ is the Sobolev best constant, thus by \eqref{Energy estimate-4} we have
\begin{align}\label{Energy estimate-5}
\lim_{n\rightarrow \infty}d^n_{13} & = \lim_{n\rightarrow \infty}\frac{1}{N}(\|u_n\|^2_1+\|v_n\|^2_3) \leq \min\left\{d_1+ \frac{1}{N}\beta_{33}^{-\frac{N-2}{2}}\mathcal{S}^{\frac{N}{2}}-\delta, ~d_3+\frac{1}{N}\beta_{11}^{-\frac{N-2}{2}}\mathcal{S}^{\frac{N}{2}}-\delta\right\}.
\end{align}
Therefore, $\{(u_n,v_n)\}$ is uniformly bounded in $H^1_0(\Omega)\times H^1_0(\Omega)$. Taking a subsequence, we have
\begin{equation*}
u_n\to u, \ v_n\to v ~\text{weakly in } H^1_0(\Omega), \text{ strongly in }  L^2(\Omega) \text{ and a.e. in }\Omega.
\end{equation*}
Since $0<\sigma_0<|\beta^n_{13}|<\sigma_1$, by \eqref{Energy estimate-6-2} and Fatou's Lemma we know that
\begin{equation}\label{Energy estimate-6-4}
\int_{\Omega}u^pv^p=0.
\end{equation}

Note that $u\geq 0, v\geq 0$ in $\Omega$, and $(u,v)$ is a solution of \eqref{two equations} with $\beta_{13}=\beta_0$. By the maximum principle we know that $u\equiv 0$ or $u>0$ in $\Omega$, and $v\equiv 0$ or $v>0$ in $\Omega$. We claim that $u>0$, $v>0$ in $\Omega$. By contradiction, without loss of generality we assume that $u\equiv 0$.

If both $u\equiv 0$ and $v\equiv 0$, then we see from $\int_\Omega v_n^2\rightarrow 0$ and $\beta^n_{13}<0$ that
\begin{equation*}
\int_\Omega |\nabla v_n|^2+\lambda_3v_n^2\leq \int_\Omega \beta_{33}|v_n|^{2p}\leq\beta_{33} \mathcal{S}^{-p}\left(\int_\Omega|\nabla v_n|^2\right)^p.
\end{equation*}
Similarly to the proof of Theorem \ref{asymtotic behavior-1} we see that
\begin{equation}\label{Energy estimate-6}
\lim_{n\rightarrow\infty}\int_\Omega |\nabla v_n|^2\geq \beta_{33}^{-\frac{N-2}{2}}\mathcal{S}^{\frac{N}{2}},
\end{equation}
and
\begin{equation}\label{Energy estimate-7}
\lim_{n\rightarrow\infty}\int_\Omega |\nabla u_n|^2\geq \beta_{11}^{-\frac{N-2}{2}}\mathcal{S}^{\frac{N}{2}}.
\end{equation}
It follows from \eqref{Energy estimate-6} and \eqref{Energy estimate-7} that
\begin{equation*}
\lim_{n\rightarrow\infty}d^n_{13}\geq \frac{1}{N}\left(\beta_{33}^{-\frac{N-2}{2}}\mathcal{S}^{\frac{N}{2}}+\beta_{11}^{-\frac{N-2}{2}}\mathcal{S}^{\frac{N}{2}}\right),
\end{equation*}
it is a contradiction with \eqref{Energy estimate-5} due to $d_i\leq \frac{1}{N}\beta_{ii}^{-\frac{N-2}{2}}\mathcal{S}^{\frac{N}{2}},i=1,3.$

If $u\equiv 0$, $v>0$ in $\Omega$, then we know that \eqref{Energy estimate-7} holds. Similarly to the proof of Theorem \ref{asymtotic behavior-1} we know that
\begin{equation}\label{E-e-8}
\int_\Omega (|\nabla v|^2+\lambda_3v^2)\geq Nd_3.
\end{equation}
Combining this with \eqref{Energy estimate-5} and \eqref{Energy estimate-7} we have
\begin{align*}
d_3+\frac{1}{N}\beta_{11}^{-\frac{N-2}{2}}\mathcal{S}^{\frac{N}{2}}-\delta & \geq \lim_{n\rightarrow\infty}d^n_{13}=\lim_{n\rightarrow\infty}\frac{1}{N}(\|u_n\|^2_1+\|v_n\|^2_3) \\
   & =\frac{1}{N}\int_\Omega (|\nabla v|^2+\lambda_3v^2)+\lim_{n\rightarrow\infty}\frac{1}{N}\int_\Omega |\nabla u_n|^2+\lim_{n\rightarrow\infty}\frac{1}{N}\int_\Omega |\nabla (v_n-v)|^2\\
   &\geq d_3+\frac{1}{N}\beta_{11}^{-\frac{N-2}{2}}\mathcal{S}^{\frac{N}{2}},
\end{align*}
which is a contradiction. It follows that $u>0$, $v>0$ in $\Omega$, and so
\begin{equation*}
\int_\Omega u^pv^p>0,
\end{equation*}
which contradicts with \eqref{Energy estimate-6-4}. Therefore, there exists a constant $C_3$ such that
\begin{equation*}
\int_{\Omega}|\beta_{13}|u_1^pu_3^p\geq C_3.
\end{equation*}
On the other hand, $\int_{\Omega}\beta_{11}u_1^{2p}\leq C$. Hence,
\begin{equation*}
\frac{\int_{\Omega}|\beta_{13}|u_1^pu_3^p}{\int_{\Omega}\beta_{11}u_1^{2p}}\geq C_4. \qedhere
\end{equation*}
\end{proof}

Set
\begin{equation*}\label{Maxi-5}
\widetilde{\beta}_1:=\sigma_1 \frac{p-1+C_4}{pC_4}\left(\frac{M_3}{m_1}\right)^{\frac{1}{2}},\quad
\widetilde{\beta}_2:=\sigma_1\frac{p-1+C_6}{pC_6}\left(\frac{\widetilde{M}_3}{\widetilde{m}_2}\right)^{\frac{1}{2}},
\end{equation*}
where $C_4$ and $C_6$ are defined in Lemma \ref{uniform bounded-2}, $M_i,\widetilde M_i$ are defined in \eqref{Maximum-15},
Set
\begin{equation}\label{Constant9}
\widehat{\beta}:=\max\{\widetilde{\beta}_1, \widetilde{\beta}_2\}.
\end{equation}

\begin{lemma}\label{Energy estimate9}
If $\beta_{12}>\widehat{\beta}$ and $\beta_{13},\beta_{23}\in [-\sigma_1,-\sigma_0]$, then
\begin{equation*}\label{energy inequality-4-2}
c< \min\{d_{13}, d_{23}\}.
\end{equation*}
\end{lemma}

\begin{proof}[\bf Proof]
We just prove that $c<d_{13}$ (which just uses the fact that $\beta_{12}>\widetilde{\beta}_1$), and the other inequality is analogous. Note that $d_{13}$ is attained by $(u_1,u_3)$, and that $(u_1,0,u_3)\in \mathcal{N}$.

{\bf Step 1 }
Given $\omega\in H^1_0(\Omega)$ with $\omega\geq 0$, we claim that there exist $C^1$ functions $t(s)=t_\omega(s),m(s)=m_\omega(s)$ and $\eta_1>0$ such that $(t(s)u_1, t(s)s^{\frac{1}{p}}\omega, m(s)u_3)\in \mathcal{N}$ for $s\in [0, \eta_1)$, and $t(0)=m(0)=1$.
Consider the $C^{1}(\mathbb{R}^3;\mathbb{R})$--functions
\begin{align*}
f(t,m,s):=&t^2\|u_1\|_1^2+t^2|s|^{\frac{2}{p}-1}s\|\omega\|_2^2 - t^{2p}\int_{\Omega}\beta_{11}u_1^{2p}- t^{2p}s^{2}\int_{\Omega}\beta_{22}\omega^{2p}-t^{2p}s\int_{\Omega}2\beta_{12}u_1^{p}\omega^{p}\\
 & -t^pm^p\int_{\Omega} \beta_{13}u_1^{p}u_3^{p}-t^pm^ps\int_{\Omega} \beta_{23}\omega^{p}u_3^{p},\\
g(t,m,s):=&m^2\|u_3\|_3^2-m^{2p}\int_{\Omega}\beta_{33}u_3^{2p}-t^pm^p\int_{\Omega} \beta_{13}u_1^{p}u_3^{p}-t^pm^ps\int_{\Omega} \beta_{23}\omega^{p}u_3^{p}.
\end{align*}
Set $\mathbf{a}_1=(1,1,0)$ and $f(\mathbf{a}_1)=g(\mathbf{a}_1)=0$. Throughout this proof, to simplify notation, we define:
\begin{equation*}
D_{ij}:=\int_{\Omega} \beta_{ij}u_i^{p}u_j^{p}, ~i,j=1,3, \quad D_{12}:=\int_{\Omega} \beta_{12}u_1^{p}\omega^{p}>0,
\quad D_{23}:=\int_{\Omega} \beta_{23}\omega^{p}u_3^{p}<0.
\end{equation*}
 Note that $\|u_1\|_1^2=D_{11}+D_{13}$ and  $\|u_3\|_3^2=D_{33}+D_{13}$. Thus, a direct computation yields that
\begin{align*}
&\frac{\partial f}{\partial t}(\mathbf{a}_1)=(2-2p)D_{11}+(2-p)D_{13}, \qquad\frac{\partial f}{\partial m}(\mathbf{a}_1)= -pD_{13},\\
&\frac{\partial g}{\partial t}(\mathbf{a}_1)=-pD_{13},  \qquad \frac{\partial g}{\partial m}(\mathbf{a}_1)= (2-2p)D_{33}+(2-p)D_{13}.
\end{align*}
Combining those with $D_{11}, D_{33}>|D_{13}|$, we obtain
\begin{equation*}
\mathcal{D}:=\frac{\partial f}{\partial t}(\mathbf{a}_1)\frac{\partial g}{\partial m}(\mathbf{a}_1)-\frac{\partial f}{\partial m}(\mathbf{a}_1)\frac{\partial g}{\partial t}(\mathbf{a}_1)>0.
\end{equation*}
Therefore, by the Implicit Function Theorem, there exist $t(s), m(s)$ of class $C^1$ in $(-\eta_1, \eta_1)$ such that $t(0)=m(0)=1$ and $f(t(s),g(s),s)=g(t(s),m(s),s)=0$. If we restrict to $s\in [0, \eta_1)$, then $(t(s)u_1, t(s)s^{\frac{1}{p}}\omega, m(s)u_3)\in \mathcal{N}$. Next, we compute $t'(0)$ and $m'(0)$: they solve the linear system
\begin{equation*}
 -\frac{\partial f}{\partial t}(\mathbf{a}_1) t'(0)-\frac{\partial f}{\partial m}(\mathbf{a}_1)m'(0)=-2D_{12}-D_{23},\quad -\frac{\partial g}{\partial t}(\mathbf{a}_1) t'(0)-\frac{\partial g}{\partial m}(\mathbf{a}_1)m'(0)=-D_{23},
\end{equation*}
which has a unique solution given by:
\begin{equation*}
t'(0)=\frac{2A}{\mathcal{D}}, \quad m'(0)=\frac{2B}{\mathcal{D}},
\end{equation*}
where
\begin{align}\label{A-define}
 A  &=  |D_{23}| \left( (p-1) D_{33} + |D_{13}|\right)- D_{12} \left( 2(p-1) D_{33}+(2-p)|D_{13}| \right), \\
B&= |D_{23}| \left((p-1) D_{11} + |D_{13}|\right)-pD_{12} |D_{13}|.
\end{align}

\noindent {\bf Step 2} We claim that  $m'(0)\leq 0\implies t'(0)< 0$. Indeed, if $m'(0)\leq 0$ then
\begin{equation}\label{Derivation-2}
|D_{23}|\leq \frac{pD_{12}|D_{13}|}{(p-1)D_{11}+|D_{13}|}.
\end{equation}
Note that $|D_{13}|<D_{11}, D_{33}$.
Combining this with \eqref{A-define} and \eqref{Derivation-2}, we obtain
\begin{align*}
  A& \leq \frac{pD_{12}|D_{13}|}{(p-1)D_{11}+|D_{13}|} \left((p-1)D_{33}+|D_{13}|\right)-D_{12} \left(2(p-1)D_{33}+(2-p)|D_{13}|\right)\\
	&=D_{12}\left((p-1)D_{33}+|D_{13}|\right)\left(\frac{p|D_{13}|}{(p-1)D_{11}+|D_{13}|}-\frac{2(p-1)D_{33}+(2-p)|D_{13}|}{(p-1)D_{33}+|D_{13}|}\right)\\
	&<D_{12}\left( (p-1)D_{33}+|D_{13}|\right) \left(1-\frac{2(p-1)D_{33}+(2-p)|D_{13}|}{(p-1)D_{33}+|D_{13}|}\right)<0
\end{align*}
that is, $t'(0)< 0$.\medbreak

\noindent {\bf Step 3} We claim that there exists $\omega\in H^1_0(\Omega)$ nonnegative such that $m'(0)=m_\omega'(0)< 0$. By contradiction, we assume that, for every $\omega\in H_0^1(\Omega)\setminus\{0\}$ nonnegative,
\begin{equation*}
\int_{\Omega} |\beta_{23}|\omega^{p}u_3^{p}\geq  \frac{p\int_{\Omega} \beta_{12}u_1^{p}\omega^{p}|D_{13}|}{(p-1)D_{11}+|D_{13}|}.
\end{equation*}
Hence,
\begin{equation*}
|\beta_{23}|u_3^{p}\geq \beta_{12}\frac{p|D_{13}|}{(p-1)D_{11}+|D_{13}|}u_1^{p},
\end{equation*}
that is
\begin{equation*}\label{negative-4-2}
u_3> \left(\frac{\beta_{12}}{|\beta_{23}|}\right)^{\frac{1}{p}}\overline{C}u_1 ~\text{ in } \Omega,  ~\text{ where } \overline{C}^p:= \frac{p|D_{13}|}{(p-1)D_{11}+|D_{13}|}=f\left(\frac{|D_{13}|}{D_{11}}\right),
\end{equation*}
for the function $f(x):=\frac{px}{p-1+x}$, which is increasing on $[0, +\infty)$. Then,
by Lemma \ref{uniform bounded-2}, we have
\begin{equation*}
\overline{C}^p\geq \frac{pC_4}{p-1+C_4}, \text{ which yields that }
u_3> \left(\frac{\beta_{12}}{|\beta_{23}|}\frac{pC_4}{p-1+C_4}\right)^{\frac{1}{p}}u_1.
\end{equation*}
Therefore, raising both sides to the power $2p$ and integrating:
\begin{equation*}
m_1\leq \int u_1^{2p}\leq\left(\frac{|\beta_{23}|}{\beta_{12}}\frac{p-1+C_4}{pC_4}\right)^2 \int_\Omega u_3^{2p}  \leq \left(\frac{|\beta_{23}|}{\beta_{12}}\frac{p-1+C_4}{pC_4}\right)^2M_3 \implies \beta_{12}\leq\sigma_1 \frac{p-1+C_4}{pC_4}\left(\frac{M_3}{m_1}\right)^{\frac{1}{2}},
\end{equation*}
which is a contradiction.\medbreak

\noindent {\bf Step 4} We complete the proof. Take $\omega$ as in Step 3. Then by Step 1 we have $t(s)=1 + t'(0)s +o(s^2)$, $m(s)=1 + m'(0)s +o(s^2)$, so
\begin{equation*}\label{energy-4}
t^2(s)=1 + 2t'(0)s +o(s^2), \quad m^2(s)=1 + 2m'(0)s +o(s^2),
\end{equation*}
and $(t(s)u_1, t(s)s^{\frac{1}{p}}\omega, m(s)u_3)\in \mathcal{N}$. Under the assumptions, we deduce from Step 2 and Step 3 that
\begin{equation*}
m'(0)< 0 ~\text{ and } ~t'(0)<0.
\end{equation*}
Combining these with $1<p<2$, we see that
\begin{align*}
c\leq & J_{12,3}(t(s)u_1, t(s)s^{\frac{1}{p}}\omega, m(s)u_3)  = \left(\frac{1}{2}-\frac{1}{2p}\right)\left(t^2(s)\|u_1\|_1^2+t^2s^{\frac{2}{p}}\|\omega\|_2^2+ m^2(s)\|u_3\|_3^2\right) \\
=& \frac{1}{N}\Big[\left(1 + 2t'(0)s +o(s^2)\right)\|u_1\|_1^2 +\left(1 + 2m'(0)s +o(s^2)\right)\|u_3\|_3^2 + \left(1 + 2t'(0)s +o(s^2)\right)s^{\frac{2}{p}}\|\omega\|^2 \Big]\\
=& \frac{1}{N}\left(\|u_1\|_1^2+\|u_3\|_3^2+2t'(0)s\|u_1\|_1^2+2m'(0)s\|u_3\|_3^2+s^{\frac{2}{p}}\|\omega\|^2 +o(s^2)\right)\\
=& d_{13} +\frac{2}{N}t'(0)\|u_1\|_1^2s +\frac{2}{N}m'(0)\|u_3\|_3^2s+\frac{1}{2N}\|\omega\|_2^2s^{\frac{2}{p}} +o(s^2)\\
<& d_{13} \quad \text{ for } s>0 \text{ small enough. } \qedhere
\end{align*}
\end{proof}
\begin{proof}[\bf Proof of Theorem \ref{existence-positive-two groups2}]
By Theorem \ref{Theorem-1}, we know that $c$ is achieved by a nonnegative solution $(\widehat{u}_1,\widehat{u}_2,\widehat{u}_3)$. Moreover, it follows from Lemma \ref{Energy estimate9} that $(\widehat{u}_1,\widehat{u}_2,\widehat{u}_3)$ is a least energy positive solution.
\end{proof}

\bibliographystyle{plain}
\end{document}